\documentclass[10pt]{amsart}

\makeatletter
\def\input@path{{../}{./}}
\makeatother
\usepackage{amsmath,amssymb,amsfonts,amsthm,bm,mathrsfs}

\usepackage{algorithm}
\usepackage{algpseudocode}


\newtheorem{example}{Example}[section]
\newtheorem{thm}{Theorem}[section]
\newtheorem{theorem}[thm]{Theorem}
\newtheorem{lemma}[thm]{Lemma}
\newtheorem{corollary}[thm]{Corollary}

\newtheorem{definition}[thm]{Definition}
\theoremstyle{remark}
\newtheorem{remark}[thm]{Remark}
\numberwithin{equation}{section}

\usepackage{graphicx}
\graphicspath{{./figs/}{./figures/}}
\usepackage{subfigure}

\usepackage[pdftex,dvipsnames]{xcolor}

\makeatletter
\newcommand{\tnorm}{\@ifstar\@tnorms\@tnorm}
\newcommand{\@tnorms}[1]{%
\left|\mkern-2.5mu\left|\mkern-2.5mu\left|
#1
\right|\mkern-2.5mu\right|\mkern-2.5mu\right|
}
\newcommand{\@tnorm}[2][]{%
\mathopen{#1|\mkern-2.5mu#1|\mkern-2.5mu#1|}
#2
\mathclose{#1|\mkern-2.5mu#1|\mkern-2.5mu#1|}
}
\makeatother

\newcommand{\jump}[1]{\lbrack\hspace{-1.5pt}\lbrack {#1} 
\rbrack\hspace{-1.5pt}\rbrack }
\newcommand{\jumpat}[2]{\lbrack\hspace{-1.5pt}\lbrack {#1} 
\rbrack\hspace{-1.5pt}\rbrack_{\raisebox{-2pt}{\scriptsize$#2$}} }

\newcommand{\dd}{\,{\rm d}}

\newcommand{\uH}{u_{\mathcal{T}}}
\newcommand{\vH}{v_{\mathcal{T}}}

\newcommand{\RV}[1]{\textcolor{black}{#1}}

\usepackage{hyperref}
\hypersetup{
	plainpages=false,
    colorlinks,
    linktocpage=true,   %
    linkcolor={red!50!black},
    citecolor={blue!50!black},
    urlcolor={blue!80!black}
} 
\usepackage[hyperpageref]{backref}

\usepackage{xcolor}

\title[Convergence and Optimality of an AmWG]
{Convergence and optimality  of an adaptive modified weak Galerkin finite element method}

\author{Yingying Xie}
\address{School of Mathematics and Information Science, Guangzhou University, Guangzhou, China, 510006.}
\email{xieyy@gzhu.edu.cn}

\author{Shuhao Cao}
\address{Division of Computing, Analytics, and Mathematics, School of Science and Engineering, University of Missouri-Kansas City, Kansas City, MO, 64110}
\email{scao@umkc.edu}

\author{Long Chen}
\address{Department of Mathematics, University of California, 
    Irvine, Irvine, CA 92697.}
\email{chenlong@math.uci.edu}

\author{Liuqiang Zhong}
\address{School of Mathematics Sciences, South China Normal University, Guangzhou, China, 510631.}
\email[Corresponding author]{zhong@scnu.edu.cn}

\thanks{The first and fourth authors were supported in part by the National Natural Science Foundation of China (No.~12071160) and the Natural Science Foundation of Guangdong Province, China(No.~2019A1515010724). The first author was also supported in part by the National Natural Science Foundation of China (No.~12101147).  The second author was supported in part by the National Science Foundation under grants DMS-1913080 and DMS-2136075. The third author was supported in part by the National Science Foundation under grants DMS-2012465, and DMS-2136075.}



\begin{document}

\begin{abstract}
An adaptive modified weak Galerkin method (AmWG) for an elliptic problem is studied in this paper, in addition to its convergence and optimality. The modified weak Galerkin bilinear form is simplified without the need of the skeletal variable, and the approximation space is chosen as the discontinuous polynomial space as in the discontinuous Galerkin method. Upon a reliable residual-based \emph{a posteriori} error estimator, an adaptive algorithm is proposed together with its convergence and quasi-optimality proved for the lowest order case. The primary tool is to bridge the connection between the modified weak Galerkin method and the Crouzeix-Raviart nonconforming finite element. Unlike the traditional convergence analysis for methods with a discontinuous polynomial approximation space, the convergence of AmWG is penalty parameter free. Numerical results are presented to support the theoretical results.
\end{abstract}
\maketitle

\section{Introduction}

Consider the following model second-order elliptic problem 
\begin{equation}
\label{eq:pb-model}
\begin{aligned}
-\nabla\cdot(A\nabla u) &= f  \quad  \mbox{in} \ \Omega,
\\
 u &= 0  \quad   \mbox{on} \ \partial\Omega,
\end{aligned}
\end{equation}
where $\Omega$ is a bounded polygonal or polyhedral domain in $\mathbb{R}^d ,d = 2,3$. Assume that an initial conforming partition $\mathcal{T}_0$ of $\Omega$ exists and for all $\tau\in\mathcal{T}_0$, the coefficient $A$ is assumed to be a piecewise constant with respect to this partition. 

Weak Galerkin (WG) is a novel numerical method for solving partial differential 
equations in which classical differential operators (such as gradient, divergence, curl) 
are approximated in a weak sense. WG method was initially introduced 
in~\cite{Wang;Ye:2013, Wang;Ye:2014Galerkin, Mu;Wang;Ye:2015} for the 
second-order elliptic problem.
Since then, the WG method has successfully found its way to many applications, for example, elliptic interface
problems~\cite{Mu;Wang;Wei;Zhao:2013}, Helmholtz equations~\cite{Mu;Wang;Ye;Zhao:2014, Mu;Wang;Ye:2015Helmholtz, Du;Zhang:2017}, biharmonic equations~\cite{Mu;Wang;Ye:2014biharmonic, Mu;Wang;Ye;Zhang:2014biharmonic}, Navier-Stokes equations~\cite{Liu;Li;Chen:2018, Hu;Mu;Ye:2019}, electromagnetic problems \cite{2015MuWangYeEtAlweak,2017ShieldsLiMachorroWeak,2022CaoWangWangnew}, and its solvers \cite{2015ChenWangWangEtAlauxiliary,2016LiXieBPX}, etc. 
In particular, Wang et al.~\cite{2014WangMalluwawaduGaoEtAlmodified} introduced a modified weak Galerkin (mWG) method for the Poisson equation. The mWG method has been successfully applied to, such as parabolic problem \cite{Gao;Wang;2014}, Signorini and obstacle problem~\cite{Zeng;Chen;Wang:2017}, and Stokes equations~\cite{Tian;Zhai;Zhang:2018}. More recently, Cui et al. generalized the mWG to biharmonic problems~\cite{2021CuiYeZhangmodified}; Li et al. showed the mWG is robust for singularly perturbed reaction-diffusion problems~\cite{2022LiChenHuangrobust}; Wang et al. presented an mWG method in a mixed form in \cite{2023WangMengZhangEtAlmodified}. For other contributions in mWG variants, we also refer the readers to \cite{2018LiChenHuangnew,2019BogrekWangSuperconvergence,2022GuoShengWangEtAlmodified,2022HussainWangAlTaweelstudy}.

The solution to \eqref{eq:pb-model} may contain singularities. To approximate problem \eqref{eq:pb-model} efficiently, the general practice is to adopt adaptivity by designing an adaptive finite element cycle through the help of the \emph{a posteriori} error estimators, a bulk marking strategy, and certain local refinement techniques.
As examples non-convergent adaptive algorithms~\cite{Chen;Dai:2002efficiency} may fail to produce the desired approximation even with additional iterations, the convergence analysis of an adaptive algorithm is of fundamental importance for ensuring that the correct approximation is obtained. It theoretically guarantees that the correct approximation will be obtained, especially if one wants to avoid the situation when more computational resources may go wasted after iterative refinements. 

The convergence theory of adaptive finite element methods is relatively mature, see~\cite{NochettoSiebert09:409} and the references therein. Nevertheless, few research results exist for the \emph{a posterior} error estimates for WG methods. Chen et al.~\cite{Chen;Wang;Ye:2014posteriori} presented the \emph{a posteriori} error estimates for second-order elliptic problems; Zhang and Chen~\cite{Zhang;Chen:2018} proposed a residual-type error estimator and proved
global upper and lower bounds of the WG method for second-order elliptic problems in a discrete $H^1$-norm; 
Li et al.~\cite{Li;Mu;Ye:2019} introduced a simple \emph{a posteriori} error estimator which can be applied to general meshes such as hybrid, polytopal and those with hanging nodes for second-order elliptic problems; Mu~\cite{Mu:2019} presented an \emph{a posteriori} error estimate for the second-order elliptic interface problems; 
Zheng and Xie~\cite{Zheng;Xie2017} discussed a residual-based \emph{a posteriori} error estimator for the Stokes problem. 
There are only a few research results for \emph{a posteriori} error estimates for mWG methods. Zhang and Lin~\cite{Zhang;Lin:2017posteriori} proposed an \emph{a posteriori} error estimator for the second-order elliptic problems. Tang et al. proposed an adaptive mWG for $\bm{H}(\mathbf{curl})$-elliptic problems in \cite{2022TangEtAlmodified}.

This paper aims to prove the optimal convergence of an AmWG algorithm for the second-order elliptic problem \eqref{eq:pb-model}. 
Different from the mWG originally introduced in~\cite{2014WangMalluwawaduGaoEtAlmodified}, we simplify the mWG as follows: for a weak function $v=\{v_0, v_b\}$, the edge/face term is not independent anymore as we choose $v_b = Q_b\{v_0\}$, i.e., $v_b$ is obtained through averaging the interior discontinuous variable $v_0$ and then projected through $Q_b$ to a one-degree-lower polynomial space. 
Compared with interior-penalty discontinuous Galerkin(IPDG) (see e.g.\cite{Arnold:1982}), the mWG is stable without choosing a sufficiently large penalty parameter. This simplification ($Q_b\{v_0\}$ opposed to $\{v_0\}$ in~\cite{2014WangMalluwawaduGaoEtAlmodified}) brings extra difficulty to the analysis of convergence. The reason is that a simple port of the workflow presented in~\cite{Bonito;Nochetto:2010Quasi-optimal}, by decomposing the discontinuous approximation space into a continuous subspace and its orthogonal complement, will introduce a penalty parameter that is not originally in the mWG discretization (c.f. \cite{2022XieEtAlConvergence}). To our best knowledge, there is no literature on the convergence of adaptive mWG methods with the skeletal variable being one degree lower than the internal variable.

To conquer this difficulty, by introducing an interpolation operator $I_\mathcal{T}$ onto the Crouzeix-Raviart 
type nonconforming finite element 
space $V^{\rm nc}(\mathcal{T})$, we bound the stabilization term and prove an \emph{a posteriori} error estimate in the discrete $H^1$-norm. One main ingredient in the convergence analysis of a standard adaptive procedure is the
orthogonality of the error to the finite element space. However, such an 
orthogonality does not hold for mWG approximations. Instead, a quasi-orthogonality result is established. Hu and Xu~\cite{Hu;Xu;2013} defined a canonical interpolation operator for the lowest Crouzeix-Raviart type 
nonconforming finite element space and established the quasi-orthogonality 
property for both the velocity and the pressure in the Stokes
problem. The main observation is that the modified weak gradient for a  
function $v_{\mathcal{T}}=\{v_0, Q_b\{v_0\}\}$ is equal to the elementwise gradient of the interpolant $I_\mathcal{T}^{^{\rm CR}} v_{0}$, namely 
$\nabla_{w} \vH=\nabla_h I_{\mathcal{T}}^{^{\rm CR}} v_0$ and we can derive the desired 
quasi-orthogonality property for the lowest order ($P_1$-$P_0$) mWG.  

Another key component to establish the optimality of the adaptive algorithm is the
localized discrete upper bound for the \emph{a posteriori} error estimator. By using a prolongation operator
introduced in~\cite{Hu;Xu;2013}, we are able to derive the discrete reliability and use it to prove the optimality of the convergence.

For the \emph{a posteriori} error analysis of mWG approximations, we mainly follow Bonito and Nochetto~\cite{Bonito;Nochetto:2010Quasi-optimal} and Chen, Wang and Ye~\cite{Chen;Wang;Ye:2014posteriori}. 
For the analysis of the convergence and the optimality of adaptive procedure, we mainly use
the Hu and Xu~\cite{Hu;Xu;2013} and Huang and Xu~\cite{Huang;Xu;2013}. We do not claim any originality on the proof of convergence and optimality.  
Instead, the main contribution of this paper is to bound the stabilization term by the element-wise residual and flux jump, as well as to establish a quasi-orthogonality and a discrete upper bound which are important ingredients on the convergence theory of adaptive finite element methods.

The rest of this paper is organized as follows. In Section \ref{se:preliminary}, the definitions of weak gradient and discrete weak gradient are introduced, as well as the modified weak Galerkin finite element spaces and the corresponding bilinear form $a_{\mathcal{T}}(\cdot, \cdot)$. 
In Section \ref{sec:error}, a residual-type error estimator is constructed, and its reliability and efficiency are shown. 
In Section \ref{sec:convergence}, we introduce an adaptive modified weak Galerkin method (AmWG) and prove its convergence and optimality. Some numerical examples are presented in Section \ref{sec:numerics} to verify the theoretical results.

\section{Notation and Preliminary}
\label{se:preliminary}
The goal of this section is to present the modified weak Galerkin (mWG) formulation for 
\eqref{eq:pb-model}. 
First, the standard weak Galerkin method is reviewed, then an mWG finite element space and the discretization thereof are introduced.

\subsection{Weak Galerkin Methods}
Given a polygonal/polyhedral element $K$ with boundary $\partial K$, the notation $v = 
\{v_0, v_b\}$ defines a weak function on $K$ such that $v_0\in L^2(K)$ and $v_b\in 
L^2(\partial K)$. Subsequently, the weak function space on $K$ is defined as
\begin{eqnarray*}
W(K) = \{v=\{v_0, v_b\}: v_0\in L^2(K), v_b\in L^2(\partial K)\}.
\end{eqnarray*}

Let $P_\ell(K)$ be the set of polynomials on $K$ with degree no more than $\ell$ ($\ell\geqslant 
1$). The discrete weak gradient operator $\nabla_{w,K,\ell} (\cdot)$ is 
defined for polynomial functions. With slight abuse of notation, when no ambiguity 
arises, we shall denote $\nabla_{w, K} (\cdot)$  as $\nabla_{w} (\cdot)$, which 
should be clear from the context. 

\begin{definition}[{\cite[Definition 2.1]{Mu;Wang;Ye:2015Galerkin}}]
\label{Def:WG}
The discrete weak gradient operator, denoted by $\nabla_{w, K, \ell}$, is defined as the 
unique polynomial $\nabla_{w, K, \ell} v \in  (P_{\ell-1} (K))^d$ satisfying the following equation for $v\in W(K)$
\begin{equation}
\big( \nabla_{w,K, \ell} v, \bm{q} \big)_K = -(v_0, \nabla\cdot \bm{q})_K + \langle v_b, \bm{q}\cdot 
\boldsymbol{n} \rangle_{\partial K},  \quad \forall q\in (P_{\ell-1} (K))^d.
\end{equation}
\end{definition}
In the definition above, $\boldsymbol{n}$ is the outward normal direction 
to $\partial K$, $(v_0, 
\nabla\cdot \bm{q})_K =\int_K v_0(\nabla\cdot \bm{q}) \, \mathrm{d} K$ is the $L^2(K)$-inner
product of $v_0$ and $\nabla\cdot \bm{q}$, and $\langle v_b, \bm{q}\cdot 
\boldsymbol{n} \rangle_{\partial K}=\int_{\partial K} v_b(\bm{q}\cdot 
\boldsymbol{n}) \, \mathrm{d} s$ is the $L^2(\partial K)$-inner produce of $v_b$ and $\bm{q}\cdot 
\boldsymbol{n}$. The differential operators involved are well-defined when restricted 
to one element. In the context of the gradient operator defined across multiple elements in $\mathcal{T}$, the elementwise gradient $\nabla_h$ is introduced, i.e., $(\nabla_h v)|_{K} := \nabla (v|_{K})$ in element $K\in \mathcal{T}$. 

In the rest of the paper, we restrict ourselves to a shape-regular triangulation $\mathcal{T}$ of $\Omega$. $\mathcal{E}$ denotes the set of all the edges or faces in $\mathcal{T}$, and $\mathcal{E}^{\mathrm{int}}$ is the set of all the interior edges or faces. Denote $|\cdot|$ the $d$-dimensional Lebesgue measure.  For $\tau\in\mathcal{T}$, its associated patch as $\omega(\tau) = \bigcup\limits_{\bar{\tau}'\cap\bar{\tau}\not=\varnothing}\tau^{\prime}$. 
For a set $\mathcal{R}\subseteq\mathcal{T}$,  its associated patch element patch is $\omega(\mathcal{R}) = \bigcup\limits_{\tau\in\mathcal{R}}\omega(\tau)$.

Given a positive integer $\ell\geqslant 1$, the $\ell$-th order weak Galerkin finite element space on $\mathcal{T}$ is defined as follows:
\begin{equation}
V^{_{\rm WG}}(\mathcal{T}):=\{v =\{v_0, v_b\}: v_0|_\tau\in P_\ell(\tau), v_b|_e\in 
P_{\ell-1}(e), e\in \mathcal E, \tau\in\mathcal{T}\},
\end{equation}
and the one with zero boundary condition:
\begin{equation}
V^{_{\rm WG}}_0(\mathcal{T})=\{v: v\in V^{_{\rm WG}}(\mathcal{T}), \ v_b=0\ 
\mbox{on}\ 
\partial\Omega\}.
\end{equation}
Note that $v_b$ is single-valued on each $e$. 

For $v=\{v_0, v_b\}\in V^{_{\rm WG}}(\mathcal{T}), w=\{w_0, w_b\}\in 
V^{_{\rm WG}}(\mathcal{T})$, the discrete bilinear form for the variational form of problem \eqref{eq:pb-model} is defined as
\begin{equation}
\label{eq:bilinear-wg}
a^{_{\rm WG}}(v, w) := \sum_{\tau\in \mathcal{T}}(A\nabla_w v, \nabla_ww)_\tau + 
\sum_{\tau\in \mathcal{T}} h_\tau^{-1}\langle Q_bv_0-v_b, 
Q_bw_0-w_b\rangle_{\partial\tau},
\end{equation}
where the weak gradient $\nabla_w = \nabla_{w,\tau, \ell-1}$, and $Q_b$ is the $L^2$ projection from $L^2(e)$ to $P_{\ell-1}(e)$ on an $e\in \mathcal{E}$. Again when it is clear from the context, we shall omit the degree of the polynomial involved in the projection operator. 

The weak Galerkin discretization is: to find a $u_h=\{u_0, u_b\}\in 
V^{_{\rm WG}}_0(\mathcal{T})$ satisfying 
\begin{equation}\label{eq:wg}
a^{_{\rm WG}}(u_h, v) =(f, v_0), \quad \forall \, v=\{v_0, v_b\}\in V^{_{\rm WG}}_0(\mathcal{T}).
\end{equation}
The WG discretization \eqref{eq:wg} is well-posed as $a^{_{\rm WG}}(\cdot, 
\cdot)$ defines an inner product on the space $V^{_{\rm WG}}_0(\mathcal{T})$. 

For completion and the convenience of our readers, we include a short argument here to show that 
\begin{equation}
\label{eq:norm-wg}
\tnorm{v}_1= \big(a^{_{\rm WG}}(v, v) \big)^{1/2}
\end{equation}
defines a norm in 
$V^{_{\rm WG}}_0(\mathcal{T})$. Assume that $\tnorm{v}_1=0$,  then $\nabla_w 
v|_{\tau}=0$ on every element $\tau\in \mathcal{T}$ and $Q_bv_0=v_b$ on 
every $e\subset\partial\tau$, then by definition \eqref{Def:WG}
\begin{eqnarray*}
0=( \nabla_w v, \nabla v_0)_\tau
&=&-(v_0,  \Delta v_0)_\tau
 +\langle v_b,  \nabla v_0\cdot \bm{n}\rangle _{\partial\tau} 
 \\
&=&-(v_0, \Delta v_0)_\tau
 +\langle Q_b v_0, \nabla v_0\cdot \bm{n}\rangle _{\partial\tau} 
 \\
&=&-(v_0, \Delta v_0)_\tau
 +\langle v_0, \nabla v_0\cdot \bm{n}\rangle _{\partial\tau} 
 \\
&=&(\nabla v_0, \nabla v_0)_{\tau}, 
\end{eqnarray*}
which implies $\nabla v_0=0$ on $\tau$. As a result, $v$ is constant on $\tau$. Moreover, $Q_b v_0=v_b$ on each $\partial\tau$. Knowing $v_b=0$ on $\partial\Omega$. By an argument of continuation, $v_0=v_b=0$. Therefore, $\tnorm{\cdot}_1$ defines a norm and $a^{_{\rm WG}}(\cdot, \cdot)$ 
defines an inner product on the space $V^{_{\rm WG}}_0(\mathcal{T})$.

The proof above reveals that the boundary part $v_b$ can be set as a polynomial with degree one less than that of the interior one $v_0$ since its presence is only in 
$\langle v_b,  \nabla v_0\cdot \bm{n}\rangle _{\partial\tau}$.

\subsection{Modified weak Galerkin finite element}
Let $e\in \mathcal{E}^{\rm int}$ be the common edge/face shared by two elements $\tau_1, 
\tau_2\in \mathcal{T}$, and denote by $\omega({e})=\tau_{1} 
\cup \tau_{2}$.
We assume that globally each $e$ is associated with a fixed unit 
normal vector $\bm{n}_e$. When $d=2$, we can get $e$'s tangential vector by 
$\bm{t}_e = \langle n_{e,2}, -n_{e,1} \rangle$, which is obtained by rotating 
$\bm{n}_e$ clockwise by $\pi/2$. Without loss of 
generality, the $\bm{n}_e$ is assumed pointing from $\tau_1$ to $\tau_2$. Denote by 
$\boldsymbol{n}_1$ and 
$\boldsymbol{n}_2$ the 
outer unit normal vectors with respect to $\tau_1$ and $\tau_2$, respectively.  For 
a smooth enough scalar function $w$, we define its average and jump on $e$ by
\begin{eqnarray*}
&&\{w\}_e = (w|_{\tau_1} + w|_{\tau_2})/{2}, \quad \mbox{for}\ \ e\in 
\mathcal{E}^{\mathrm{int}};\\
&&\jumpat{w}{e}=w|_{\tau_1}  - w|_{\tau_2} , \quad \mbox{for}\ \ 
e\in \mathcal{E}^{\mathrm{int}};\\
&&\jumpat{w}{e}= \{w\}_e=w, \quad \mbox{for}\ \ e\subset 
\partial\Omega.
\end{eqnarray*}
Similarly, for an admissible vector function $\boldsymbol{w}$, we have
\begin{eqnarray*}
&&\{\boldsymbol{w}\}_e = (\boldsymbol{w}|_{\tau_1} + \boldsymbol{w}|_{\tau_2})/{2}, 
\ \mbox{for}\ \ e\in \mathcal{E}^{\mathrm{int}};\\
&&\jump{\boldsymbol{w} \cdot\bm{n}_e}_{e}=\boldsymbol{w}|_{\tau_1} 
\cdot\boldsymbol{n}_1 + 
\boldsymbol{w}|_{\tau_2}\cdot \boldsymbol{n}_2, \ \mbox{for}\ \ e\in 
\mathcal{E}^{\mathrm{int}};\\
&& \{\boldsymbol{w}\}_e=\boldsymbol{w}, \quad 
\jumpat{\boldsymbol{w}\cdot \bm{n}_e}{e}=\boldsymbol{w}\cdot\boldsymbol{n}_e, \ \mbox{for}\ \ 
e\subset 
\partial\Omega.
\end{eqnarray*}
In the definitions above, $w|_{\tau}$ and $\boldsymbol{w}|_{\tau}$'s values on $e$ 
are defined for those spaces which yield a well-defined trace, respectively.

Specifically, in discretization \eqref{eq:wg}, $v_0$ is chosen to be in the 
discontinuous polynomial space 
\begin{eqnarray*}  
V^{\scriptscriptstyle{\rm DG}}(\mathcal{T}) := \{ v: v|_\tau \in P_\ell(\tau), 
\tau\in\mathcal{T} \},
\end{eqnarray*}
and the edge/face term $v_b = Q_b\{v_0\}$. Now a weak function is 
$v=\{v_0, Q_b\{v_0\}\}$, which offers a continuous embedding of  $V^{\scriptscriptstyle{\rm DG}}(\mathcal{T})\hookrightarrow 
V^{_{\rm WG}}(\mathcal{T})$ with respect to norm \eqref{eq:norm-wg}, which will be shown equivalent to the modified IP norm associated with \eqref{eq:bilinear-mwg}. The modified weak Galerkin finite element space for $\mathcal{T}$ is then defined as 
\begin{equation}
\label{eq:space-mwg}
V(\mathcal{T}):=\{v_\mathcal{T}=\{v_0, v_b \} :  v_0 \in V^{\scriptscriptstyle{\rm DG}}(\mathcal{T}); v_b=Q_b\{v_0\} , e\in \mathcal{E} \},
\end{equation}
and 
\begin{equation}
V_0(\mathcal{T})=\{v_\mathcal{T}: v_\mathcal{T}\in V(\mathcal{T}), \quad Q_b\{v_0\}=0\ \mbox{on}\ 
\partial\Omega\}.
\end{equation}
\begin{remark}
The original modified weak Galerkin function space was introduced in 
\cite{2014WangMalluwawaduGaoEtAlmodified}. 
The edge/face term $v_b$ is chosen as $\{v_0\}|_e\in P_{\ell}(e)$ in 
\cite{2014WangMalluwawaduGaoEtAlmodified}, while $v_b \equiv Q_b\{v_0\} \in P_{\ell-1}(e)$ here to match the reduced-order weak Galerkin scheme in~\cite{Mu;Wang;Ye:2015Galerkin}. 
\end{remark}

The definition of the weak gradient is modified accordingly as follows.
\begin{definition}\label{Def:MWG}
The modified weak gradient operator acting on any $v_\mathcal{T}\in V(\mathcal{T})$, denoted by 
$\nabla_{w,\tau}$ on $\tau \in \mathcal T$, is defined as the unique polynomial in $(P_{\ell-1}(\tau))^d, d=2, 3$ such that its inner product with any $\bm{q}\in (P_{\ell-1}(\tau))^d$ satisfies the following equation:
\begin{equation}
\label{eq:mwg}
( \nabla_{w,\tau} v_\mathcal{T}, \bm{q})_\tau = -(v_0, \nabla\cdot \bm{q})_\tau + \langle Q_b\{v_0\}, 
\bm{q}\cdot 
\boldsymbol{n} \rangle_{\partial \tau}.
\end{equation}
When the triangulation is clear from the context, we shall abbreviate the notation $\nabla_{w,\tau}$ as $\nabla_w$. 
\end{definition}

Let $e\in \mathcal{E}^{\rm int}$ be the common edge/face shared by $\tau_1, \tau_2\in 
\mathcal{T}$. By using the mWG space \eqref{eq:space-mwg} in the WG bilinear 
form \eqref{eq:bilinear-wg}, one can simplify the stabilization term on $e$ for $v_\mathcal{T}=\{v_0,  Q_b\{v_0\}\}\in V(\mathcal{T})$ and $w_\mathcal{T}=\{w_0,  Q_b\{w_0\}\}\in V(\mathcal{T})$, henceforth the same argument applies to other edges/faces.
 \begin{eqnarray*} 
\lefteqn{  \langle Q_b(v_0-\{v_0\}), Q_b(w_0-\{w_0\})\rangle_{e}  }
\\
&&=\left\langle Q_b\left(v_0|_{\tau_1}-\frac{v_0|_{\tau_1}+v_0|_{\tau_2}}{2}\right), 
Q_b\left(w_0|_{\tau_1}-\frac{w_0|_{\tau_1}+w_0|_{\tau_2}}{2}\right)\right\rangle_{e}
\\
&&= \left\langle Q_b\left(\frac{v_0|_{\tau_1}-v_0|_{\tau_2}}{2}\right), 
Q_b\left(\frac{w_0|_{\tau_1}-w_0|_{\tau_2}}{2}\right)\right\rangle_{e}
\\
&&= \frac{1}{4} \langle Q_b\jump{v_0}{}, Q_b\jump{w_0}{}\rangle_{e}.
 \end{eqnarray*}
Consequently, for $v_\mathcal{T}, w_\mathcal{T}\in V(\mathcal{T})$, the associated bilinear form is 
defined as
\begin{eqnarray}\nonumber
a_{\mathcal{T}}(v_\mathcal{T}, w_\mathcal{T}) &:= &
\sum_{\tau\in \mathcal{T}}(A\nabla_w v_\mathcal{T}, \nabla_ww_\mathcal{T})_\tau + 
4 \sum_{\tau\in \mathcal{T}}  h_\tau^{-1}\langle Q_b(v_0-\{v_0\}), 
Q_b(w_0-\{w_0\})\rangle_{\partial\tau}
\\  \label{eq:bilinear-mwg}
&=& \sum_{\tau\in \mathcal{T}}(A\nabla_w v_\mathcal{T}, \nabla_ww_\mathcal{T})_\tau + \sum_{\tau\in 
\mathcal{T}}  h_\tau^{-1}\langle Q_b\jump{v_0}{}, Q_b\jump{w_0}{}\rangle_{\partial\tau}.
\end{eqnarray}

A modified weak Galerkin (mWG) approximation is then to seek $\uH \in 
V_0(\mathcal{T})$ satisfying 
\begin{equation}
\label{eq:pb-mwg}
a_{\mathcal{T}}(\uH, v_\mathcal{T}) =(f, v_0), \quad \forall v_\mathcal{T}  \in V_0(\mathcal{T}).
\end{equation}
It is clear that the modified weak Galerkin finite element scheme \eqref{eq:pb-mwg} is also well-posed since the problem is solved in a subspace (the embedding of the DG space) of the original WG space. If nonhomogeneous Dirichlet boundary condition $u=g$ on $\partial \Omega$ is presented, $\uH$ can be decomposed to the sum of two parts, the first part satisfies the approximation problem above. For the second part, the boundary contribution is $\{0, Q_b g\}$ which can be moved to the right-hand side in addition to the source term. For other treatments of the boundary data such as interpolations, see e.g.,~\cite{2013MuWangWangEtAlcomputational}.
\begin{remark}
\label{remark:ipdg}
In the classic IPDG formulation (see e.g.,~\cite{Bonito;Nochetto:2010Quasi-optimal}), for $ v_{0}, w_{0}\in V^{\scriptscriptstyle{\rm DG}}(\mathcal{T})$, the bilinear form is 
\begin{eqnarray}\nonumber
a_{\mathcal{T}}^{_{\rm IPDG}}(v_{0}, w_{0}) 
&=& \sum_{\tau\in \mathcal{T}}(A\nabla_h v_{0}, \nabla_h w_{0})_{\tau} - \sum_{e\in \mathcal{E}} \langle\jump{v_{0}}{}, \{A\nabla_h w_{0} \cdot \boldsymbol{n}\}\rangle_{e}
\\ \label{eq:IPDG}
&&\hspace{-0.5cm}-
\sum_{e\in \mathcal{E}} \langle\jump{w_{0}}{},\{A\nabla_h v_{0} \cdot \boldsymbol{n}\}\rangle_{e} 
+ \sum_{\tau\in 
\mathcal{T}}\mu h_\tau^{-1}\langle \jump{v_{0}}{}, \jump{w_{0}}{}\rangle_{\partial\tau}.
\ \ \ \ \ 
\end{eqnarray}

In comparison, in \eqref{eq:bilinear-mwg}, using the definition of the weak gradient, several times of the integration by parts, and the assumption that $A$ is a piecewise constant, for $v_\mathcal{T}=\{v_0,  Q_b\{v_0\}\}, 
w_\mathcal{T}=\{w_0,  Q_b\{w_0\}\}\in V(\mathcal{T})$, an equivalent bilinear form of \eqref{eq:bilinear-mwg} reads
\begin{eqnarray}\nonumber
a_{\mathcal{T}}(v_{\mathcal{T}}, w_{\mathcal{T}}) 
&=& \sum_{\tau\in \mathcal{T}}(A\nabla_h v_{0}, \nabla_h w_{0})_{\tau} 
- \sum_{e\in \mathcal{E}} \langle\jump{v_{0}},\{A\nabla_h w_{0} \cdot \boldsymbol{n}\}\rangle_{e}
\\  \label{eq:pb-mwg2}
&&\hspace{-0.5cm}- \sum_{e\in \mathcal{E}} \langle\jump{w_{0}}{}, \{A\nabla_w v_{\mathcal{T}} \cdot \boldsymbol{n}\}\rangle_{e} 
+\sum_{\tau\in 
\mathcal{T}}  h_\tau^{-1}\langle Q_b\jump{v_{0}}{}, Q_b\jump{w_{0}} \rangle_{\partial\tau}.
\ \ \ \ \ \ \
 \end{eqnarray}

The only difference is that one $\nabla_h v_{0}$ in \eqref{eq:IPDG} is replaced by $\nabla_w v_{\mathcal{T}}$ in \eqref{eq:pb-mwg2}. This minor change leads to a major improvement that the mWG \eqref{eq:pb-mwg} is automatically coercive and continuous as the bilinear form $a_{\mathcal{T}}(\cdot, \cdot)$ induces a norm. While in IPDG \eqref{eq:IPDG}, the penalty parameter $\mu$ being sufficiently large is a necessary condition to achieve the coercivity both theoretically and numerically~\cite{Shahbazi:2005explicit}.
\end{remark}

The following quantity is well-defined for $H^1_0(\Omega)+V_0(\mathcal{T})$ and 
defines a mesh dependent norm on $V_0(\mathcal{T})$ (for similar results see e.g.,
\cite{Mu;Wang;Ye:2015Galerkin})
\begin{equation}
\label{A2}
\tnorm{\vH }_{\mathcal{T}}^2:= a_{\mathcal{T}}\left(\vH, 
\vH\right)
= \bigl\|A^{1/2}\nabla_{w} \vH \bigr\|_{\mathcal{T}}^2
+  \sum_{\tau\in \mathcal{T}}   h_\tau^{-1} \big\|Q_b\jump{v_0}{}
\big\|^2_{\partial\tau}.
\end{equation}
We note that \eqref{A2} defines a norm on $V_0(\mathcal{T})$ naturally, since $V_0(\mathcal{T})$ is a subspace of the bigger space $V_0^{_{\rm WG}}(\mathcal{T})$ and $\tnorm{\vH }_{\mathcal{T}}^2$ is equivalent to \eqref{eq:norm-wg} {\RV{restricted}} to this subspace; see \eqref{eq:norm-wg} and the proof afterwards.

With slight abuse of notation, we are interested in estimating the following error 
using computable quantities and designing a convergent adaptive algorithm to reduce its magnitude successively:
\begin{equation}
\label{eq:error}
\tnorm{u - \uH}^2_{\mathcal{T}} := 
\bigl\|A^{1/2}(\nabla u - \nabla_{w} \uH) \bigr \|^2
+ \sum_{\tau\in \mathcal{T}}   h_\tau^{-1} \big\|Q_b\jump{u-u_0}{}
\big\|^2_{\partial\tau}.
\end{equation}

\section{\emph{A Posteriori} Error Analysis}
\label{sec:error}

In this section, we shall prove the reliability and efficiency of a residual-type error estimator. For $\tau\in\mathcal{T}$, we define the element-wise 
error estimator as
\begin{eqnarray}\nonumber
 \lefteqn{
\eta^2(\nabla_w \vH, \tau)  := \;
h_{\tau}^{2}A^{-1}_{\tau}\left\|R(\nabla_w \vH)\right\|_{0, \tau}^{2} 
    }  
\\ \label{eq:eta}
 &&\quad \ \   +\sum_{e\subset\partial\tau}h_{\tau}\int_{e}
 \left(\big(A_e^{\max}\big)^{-1} J^2_{n,e}(A\nabla_w \vH) 
+ A_e^{\min} J^2_{t,e}(\nabla_w \vH)\right)\,\mathrm{d}s,\ \ \
\end{eqnarray}
where $A_e^{\max} := \max\{A_{\tau_1}, A_{\tau_2} \}$  
and $A_e^{\min} := \min\{A_{\tau_1}, A_{\tau_2} \}$  for $\tau_1,\tau_2\in \omega(e)$.
The element residual is defined as
\begin{eqnarray*}
R(\nabla_w \vH) = f+\nabla\cdot\left( A\nabla_{w} 
\vH\right), 
\end{eqnarray*}
and the normal jump of the weak flux is defined as
\begin{eqnarray*}
J_{n,e}(A\nabla_w \vH) := 
\left\{ \begin{array}{l}
 \jump{A\nabla_w \vH\cdot \bm{n}_e}_{e}, \quad \mbox{if} \ \ e\in 
 \mathcal{E}^{\mathrm{int}}\\
 0,  \quad\quad\quad \quad\quad\quad \mbox{otherwise}.
     \end{array} 
     \right.
\end{eqnarray*}
For the tangential jumps, when $d=2$:
\[
J_{t,e}(\nabla_w \vH) := 
\left\{ \begin{array}{l}
 \jumpat{\nabla_w \vH\cdot \bm{t}_e}{e}, \quad \mbox{if} \ \ e\in 
 \mathcal{E}^{\mathrm{int}}\\
 0,  \quad\quad\quad \quad\quad\quad \mbox{otherwise},
     \end{array} 
     \right.
\]
and when $d=3$
\[
J_{t,e}(\nabla_w \vH) := 
\left\{ \begin{array}{l}
 \jumpat{\nabla_w \vH\times \bm{n}_e}{e}, \quad \mbox{if} \ \ e\in 
 \mathcal{E}^{\mathrm{int}}\\
 0,  \quad\quad\quad \quad\quad\quad \mbox{otherwise}.
     \end{array} 
     \right.
\]
Then the error estimator for the set $\mathcal{M} \subseteq\mathcal{T}$ is defined as
\begin{equation}
\label{eta2}
\eta^2(\nabla_w \vH, \mathcal{M}) = 
\sum_{\tau\in\mathcal{M}}\eta^2(\nabla_w \vH, \tau).
\end{equation}

In computing the error estimator, only the information of $\nabla_w \vH$ is used. Thus, we opt for a notation of $\eta(\nabla_w \vH, \mathcal{M})$ instead of $\eta(\vH, \mathcal{M})$. In Remark \ref{remark:eta-gradw}, some further explanation is given with regard to this choice of the notation for the error estimator in the context of the convergence analysis.

\subsection{A space decomposition}
In this section, a space decomposition is first introduced 
to bridge the past result of the full degree jump $\jump{v_0}{}$ to $Q_b\jump{v_0}{}$ with one less degree. Then, similar to~\cite{Bonito;Nochetto:2010Quasi-optimal}, a 
partial orthogonality for the mWG approximation $\uH$ is introduced to enable the insertion of continuous interpolants to prove the reliability.

For $V^{\scriptscriptstyle{\rm DG}}(\mathcal{T})$ consisting of discontinuous polynomial of degree $\ell$ :
\begin{equation}
\label{eq:space}
V^{\rm c}(\mathcal{T}) \subset V^{\rm nc}(\mathcal{T}) \subset V^{\scriptscriptstyle{\rm DG}}(\mathcal{T}),
\end{equation}
where $V^{\rm c}(\mathcal{T}):=V^{\scriptscriptstyle{\rm DG}}(\mathcal{T})\cap H_0^1(\Omega)$ is 
 the continuous Lagrange finite element space. $V^{\rm nc}(\mathcal{T})$ is a subspace of $V^{\scriptscriptstyle{\rm DG}}(\mathcal{T})$ so that their 
quotient 
space is endowed with the induced topology of the seminorm $\sum_{e\in \mathcal{E}} 
h_\tau^{-1} \big\|Q_b\jump{\cdot}{}
\big\|^2_{e}$.  When $\ell=1$, $V^{\rm nc}(\mathcal{T})$ is simply the well-known 
Crouzeix-Raviart finite 
element space~\cite{Crouzeix;Raviart:1973}. When $\ell\geq 1$, $V^{\rm nc}(\mathcal{T})$ is a generalization of the 
Crouzeix-Raviart type nonconforming finite element space 
\cite{Ciarlet;Dunkl;Sauter:2018family}, 
which can be viewed a special case of nonconforming virtual element 
\cite{Ayuso-de-Dios;Lipnikov;Manzini:2016nonconforming} restricted on triangulations:
\begin{equation}
\label{eq:nc}
V^{\rm nc}(\mathcal{T}):= V^{\scriptscriptstyle{\rm DG}}(\mathcal{T}) \cap H^{1, {\rm nc}}\big(\mathcal{T} \big),
\end{equation}
where
\begin{equation}
\label{eq:nc-continuity}
H^{1, \rm nc}\big(\mathcal{T} \big)=\left\{v \in 
\prod_{\tau \in \mathcal{T}} H^1(\tau): \int_{e} \jump{v}{}  q \, \mathrm{d} 
s=0 
\quad 
\forall q \in P_{\ell-1}(e), \forall e \in \mathcal{E}\right\}.
\end{equation}

Unlike the virtual element space, where the shape function may not be polynomials, it is shown in 
\cite{Ciarlet;Ciarlet;Sauter;Simian:2016Intrinsic,Ciarlet;Dunkl;Sauter:2018family}
 that the space $V^{\rm nc}(\mathcal{T})$ can be obtained from $V^{\rm 
c}(\mathcal{T})$ by adding locally supported polynomial bases. 
More specifically, in~\cite{Ciarlet;Dunkl;Sauter:2018family}, the authors constructed a set of nodal 
basis for a strict subspace of $V^{\rm nc}(\mathcal{T})$ while satisfying the 
continuity 
constraint \eqref{eq:nc-continuity}. However, to serve the purpose of this paper, 
$V^{\rm nc}(\mathcal{T})$ is to bridge the proofs, it only suffices to know the 
existence of a set of unisolvent degrees of freedom, while not explicitly construct 
the dual basis of it. We refer the reader to~\cite[Section 
3.2]{Ayuso-de-Dios;Lipnikov;Manzini:2016nonconforming} for the degrees of freedom 
for general polytopes which includes the case of triangulations ($d=2,3$) in this 
paper.

When restricted a smaller subspace 
$V^{\rm c}(\mathcal{T}) \subset V^{\scriptscriptstyle{\mathrm{DG}}}(\mathcal{T})$, the weak gradient coincides with the piecewise gradient, and the stabilization is vanished which leads to the following partial orthogonality. 
\begin{lemma}
\label{lem:PartialOrthogonality}
 Let $u$ and $\uH \in V_0(\mathcal{T})$ be the solutions of 
 \eqref{eq:pb-model} and \eqref{eq:pb-mwg}, respectively, then
 \begin{eqnarray}
 \label{eq:orthogonality-partial}
 \big(A\nabla u - A\nabla_w \uH, \nabla v^{\rm c}\big)_{\mathcal{T}} = 0, \quad \forall 
 v^{\rm c}\in V^{\rm c}(\mathcal{T}).
 \end{eqnarray} 
\end{lemma}
\begin{proof}
It follows from $v^{\rm c} \in H_0^1(\Omega)$ that $(A\nabla u, \nabla v^{\rm c}) = (f, v^{\rm c})$. In addition, $\jump{v^c}{}=0$ for functions when $v^{\rm c}\in V^{\rm c}(\mathcal{T})\subset 
H_0^1(\Omega)$, thus $Q_b\jump{v^{\rm c}}{}=0$ in $a_{\mathcal{T}}(\cdot,\cdot)$, which further implies 
$\nabla_w v^c= \nabla v^c$. 
Lastly since $V^{\rm c}(\mathcal{T})\cap H_0^1(\Omega)\subset 
V_0(\mathcal{T})$ implies 
that $a_{\mathcal{T}}(\uH, v^{\rm c})=(f, v^{\rm c})$ for any $v^{\rm c}\in 
V^{\rm c}(\mathcal{T})$, as a result, $(A\nabla_w u, 
\nabla v^{\rm c})_{\mathcal{T}} = (f, v^{\rm c})$ and the lemma follows.
\end{proof}

The following interpolation operator to the nonconforming space will play an important role in the analysis.
\begin{lemma}
\label{lem:est-nc-interp}
There exist an interpolation operator $I_{\mathcal{T}}: V^{\scriptscriptstyle{\rm DG}}(\mathcal{T})\rightarrow 
V^{\rm nc}(\mathcal{T})$, which is locally defined and a projection, 
as well as a constant depending only on the shape regularity of $\tau$ 
such that for all $\tau\in \mathcal{T}$  the following inequality hold:
for $|a| = 0, 1$,
\begin{equation}
\label{eq:est-interp}
\big\|D^a(\vH - I_{\mathcal{T}} \vH)\big\|_{\tau}^2\lesssim 
\sum_{e\in 
\mathcal{E}^{\mathrm{int}}(\omega(\tau))}h_\tau^{1-2|a|}
\bigl\|Q_b\jump{v_\mathcal{T}}{}\bigr\|^2_e, 
\quad\forall \vH\in V^{\scriptscriptstyle{\rm DG}}(\mathcal{T}).
\end{equation}
\end{lemma}

\begin{proof}
The proof follows a similar argument as the one in~\cite[Lemma 
6.6]{Bonito;Nochetto:2010Quasi-optimal}, and is presented here for completion. 
Denote the set of the degrees of freedom 
functionals in $V^{\rm nc}(\mathcal{T})$ as $\mathcal{N} = \{\gamma_i(\cdot)\}_{i\in 
\Lambda}$, and the nodal bases set is $\{\phi_i(\bm{x})\}$ corresponding to 
$\mathcal{N}$. Let $\omega_i = \operatorname{supp} 
\phi_i$, where $\omega_i$ is either a single element $\tau$ or $\omega(e)$, then we 
consider the following projection $\Pi_i \vH$ obtained through
\begin{equation}
\label{eq:nc-proj}
\big(\vH- \Pi_i \vH, w\big)_{\omega_i} = 0, \quad \forall w\in V^{\rm 
nc}(\omega_i):=  
\Pi_{\tau\in \omega_i}P_\ell(\tau)
\cap H^{1, \rm nc}\big(\mathcal{T} \big),
\end{equation}
and then the interpolation is defined as:
\begin{equation}
\label{eq:nc-interp}
I_{\mathcal{T}} \vH(\bm{x}) := 
\sum_{i\in \Lambda} \gamma_i\big(\Pi_i \vH \big) \phi_i(\bm{x}).
\end{equation}
If $\vH\in V^{\rm nc}(\omega_i)$, we have locally $\Pi_i \vH=\vH$ on $\omega_i$, 
thus 
$\gamma_i\big(\Pi_i \vH \big) = \gamma_i\big(\vH \big)$, and $I_{\mathcal{T}}$ is a 
projection. 

Moreover, $\vH\in V^{\rm nc}(\omega_i)$ implies that $\sum_{e\in 
\mathcal{E}^{\mathrm{int}}(\omega_i)}h_\tau^{-1}\big\|Q_b\jump{v_\mathcal{T}}{}\big\|^2_e = 
0$, as well as $\vH- \Pi_i \vH = 0$ by \eqref{eq:nc-proj}, hence by a scaling 
argument and the equivalence of norms on a finite-dimensional space, we have
\begin{equation}
\label{eq:nc-proj-est}
\sum_{\tau\in \omega_i} \|D^a(\vH - \Pi_i \vH)\|_{\tau}^2\lesssim 
\sum_{e\in \mathcal{E}^{\mathrm{int}}(\omega_i)}
h_\tau^{1-2|a|}\|Q_b\jump{\vH}{}\|^2_e,
\end{equation}
where we note that if a nodal basis $\phi_i$ has support $\omega_i = \tau$, both 
sides of the inequality above will be 0, as the projection \eqref{eq:nc-proj} modulo 
out the element bubbles. 

Next, consider on a $\tau\in \mathcal{T}$, and all local degrees of freedom 
$\{\gamma_i(\cdot)\}_{i=1}^{N_\tau}$ of which the nodal basis having support 
$\omega_i$ overlaps with $\tau\cup \partial \tau$, we have on $\tau$
\begin{equation}
\label{eq:nc-interp-split}
\vH - I_{\mathcal{T}}
= (\vH- \Pi_1 \vH) - \sum_{i=1}^{N_\tau} \gamma_i(\Pi_1 \vH - \Pi_i \vH) \phi_i.
\end{equation}
On $\tau\cup \partial \tau$, $\gamma_i(\vH)$ 
is either defined on $e\subset \partial \tau$ (see~\cite[Section 
3.2]{Ayuso-de-Dios;Lipnikov;Manzini:2016nonconforming})
\[
\gamma_i(\vH) = \frac{1}{|e|}\int_e \vH p \,\dd s, \forall p\in P_{\ell-1}(e),
\]
or defined on $\tau$ ($\ell\geq 2$)
\[
\gamma_i(\vH) = \frac{1}{|\tau|}\int_\tau \vH p \,\dd s, \forall p\in P_{\ell-2}(\tau),
\]
and in both cases, we have by a simple scaling argument and by \eqref{eq:nc-proj-est}
\[
\big\|\gamma_i(\Pi_1 \vH - \Pi_i \vH) \big\|_{\tau}
\lesssim \big\|\Pi_1 \vH - \Pi_i \vH \big\|_{\tau}
\lesssim \sum_{e\in \mathcal{E}^{\mathrm{int}}(\omega_i)}
h_\tau\|Q_b\jump{v_0}{}\|^2_e.
\]
Lastly, combining the estimate above with \eqref{eq:nc-interp-split}, using the 
fact that $\|\phi_i\|_{\infty,\tau} \lesssim 1$, and estimate \eqref{eq:nc-proj-est} 
yields the desired estimate.
\end{proof}

\subsection{Reliability}
In this section, the global reliability of the estimator in \eqref{eta2} is to be shown. The difference between the weak gradient and classical gradient can be controlled by the jump term which becomes a handy tool in our analysis.

\begin{lemma}
\label{lem:est-diff-weak}
For $\vH =\{v_0, v_b\}\in V(\mathcal{T})$, it holds for any $\mathcal{M}\subseteq 
\mathcal{T}$
\begin{eqnarray*}
\bigl\|A_\tau^{1/2}(\nabla_{w} \vH-\nabla_h v_0)\bigr\|_{\mathcal{M}}^{2}
\lesssim 
\sum_{e \in \mathcal{E}(\mathcal{M})} 
  A_e^{\max} h_{\tau}^{-1}\|Q_b\jump{v_0}{}\|^2_{e}.
\end{eqnarray*}
\end{lemma}

\begin{proof}
The proof follows from~\cite[Lemma 2.1]{Zhang;Lin:2017posteriori} with coefficient 
added and a more localized version is presented. On $\tau\in\mathcal{T}$, let 
$\bm{q} = \nabla_{w} \vH - \nabla_h v_0 \in 
\bm{P}_{l-1}(\tau)$, applying 
integration by parts on $\nabla_h v_0$ and the weak gradient definition  
\eqref{eq:mwg} on $\nabla_{w} \vH$, we have 
\begin{eqnarray*} 
\int_{\tau} A_\tau(\nabla_{w} \vH - \nabla_h v_0)\cdot \bm{q} 
\,\mathrm{d} x
&=&\int_{\partial 
\tau}\big(Q_b\{v_0\}-v_0\big) A_{\tau}\bm{q}  
\cdot \boldsymbol{n} \,\mathrm{d}s
\\
&=& \int_{\partial \tau} 
Q_b\big(\{v_0\}-v_0\big) 
 A_{\tau}\bm{q}  \cdot \boldsymbol{n} \,\mathrm{d}s
\\
&=& 
\sum_{e\subset\partial\tau}\pm\frac{1}{2}\int_{e}Q_b\jump{v_0}{}
 A_{\tau}\bm{q}  \cdot \boldsymbol{n} \,\mathrm{d}s,
\end{eqnarray*}
where the plus or minus sign depends on whether the outward normal for 
$e\subset\partial \tau$ coincides with the globally defined normal for this 
edge/face. Now by a standard trace inequality and an inverse inequality, we have
\begin{eqnarray*}
\lefteqn{\bigl\|A^{1/2}_\tau(\nabla_{w} \vH - \nabla_h 
v_0)\bigr\|_{\tau}^2}\\
&& \lesssim \sum_{e\subset\partial\tau}
A_{\tau}^{1/2} h_{\tau}^{-1/2}\left\|Q_b\jump{v_0}{}\right\|_{e}
A_{\tau}^{-1/2} h_{\tau}^{1/2}
\left\|A_{\tau}\boldsymbol{q}\cdot \bm{n}\right\|_{e}
\\
&&\lesssim \left(\sum_{e\subset\partial\tau} 
  A_e^{\max} h_{\tau}^{-1}\left\|Q_b\jump{v_0}{}\right\|^2_{e}\right)^{1/2} 
\bigl\|A^{1/2}_\tau(\nabla_{w} \vH - \nabla_h 
v_0)\bigr\|_{\tau},
\end{eqnarray*}
then the desired result follows by canceling a $\bigl\|A^{1/2}_\tau(\nabla_{w} 
\vH - \nabla_h v_0)\bigr\|_{\tau}$ and summing up the element-wise estimate for every 
$\tau\in \mathcal{M}$.
\end{proof}

Next we bound the stabilization term by the element-wise residual and the normal jump of the weak flux. 
We note that in~\cite{Zheng;Xie2017}, though focusing on a different 
model problem, the $h^{-1}$-weighted solution jump can be used as the sole 
error estimator to guarantee reliability and efficiency up to data oscillation. 
Nevertheless, the motivation here is to change the dependence of the error 
indicators on the local mesh size $h$ from $h^{-1}$ to $h^s, s>0$, so that a contraction 
property of the estimator can be proved in Section \ref{sec:est-eta} without any 
saturation assumptions, which is one of the keys to show the convergence of an 
adaptive algorithm.

\begin{lemma}
\label{lem:est-jump}
Let $u$ be the weak solution of \eqref{eq:pb-model} and $\uH=\{u_0, u_b\} \in 
V_0(\mathcal{T})$ be the solution to \eqref{eq:pb-mwg}, we have
\begin{eqnarray}\nonumber
\sum_{e\in\mathcal{E}} 
   A_e^{\max} h_\tau^{-1}\|Q_b\jump{u_0}{}\|_e^2
&\lesssim & \Bigg(\sum_{\tau \in 
\mathcal{T}}h_{\tau}^{2}A^{-1}_{\tau}\left\|R(\nabla_w \uH)\right\|_{0, 
\tau}^{2} 
\\ \label{eq:est-jump-eta} 
&& +\sum_{e\in\mathcal{E}}h_{\tau}(A_e^{\max})^{-1} \|J_{n,e}(A\nabla_w 
\uH) \|_{e}^2\Bigg),\ \ 
\end{eqnarray}
where the constant depends on the shape regularity of $\mathcal{T}$.
\end{lemma}

\begin{proof} 
For simplicity, we denote $A_e := A_e^{\max}$ in the proof. 
Let that $I_{\mathcal{T}} \uH\in V^{\rm nc}(\mathcal{T})$ be the 
nonconforming interpolation defined in Lemma \ref{lem:est-nc-interp}, using the 
definition of the modified weak derivative \eqref{eq:mwg} element-wisely, we have
\begin{eqnarray*} 
 \sum_{e\in\mathcal{E}}   A_e h_\tau^{-1}\|Q_b\jump{u_0}{}\|_e^2
&=& \sum_{e\in\mathcal{E}}   A_e h_\tau^{-1}\langle Q_b\jump{u_0}{}, 
Q_b\jump{\uH-I_{\mathcal{T}} \uH}{}\rangle_e
\\
&=&  (f, \uH-I_{\mathcal{T}} \uH)_{\mathcal{T}} - 
(A\nabla_w \uH, \nabla_w(\uH-I_{\mathcal{T}} 
\uH))_{\mathcal{T}}
\\
&=& (f, \uH-I_{\mathcal{T}} \uH)_{\mathcal{T}} + 
(\nabla_h\cdot (A\nabla_w \uH), \uH-I_{\mathcal{T}} 
\uH)_{\mathcal{T}} 
\\
&&- 
\sum_{\tau\in\mathcal{T}_h}\langle(A\nabla_w \uH) \cdot\bm{n}, \{\uH-I_{\mathcal{T}} 
\uH\}\rangle_{\partial 
\tau}
\\
&=& (f +\nabla_h\cdot (A\nabla_w \uH), \uH-I_{\mathcal{T}} 
\uH)_{\mathcal{T}} 
\\
&&
- 
\sum_{e\in\mathcal{E}}
\left\langle\jump{(A\nabla_w \uH) \cdot\bm{n}}{}, \{\uH-I_{\mathcal{T}} \uH\}
\right\rangle_e. 
\end{eqnarray*}
By the Cauchy-Schwarz inequality, we have
\begin{eqnarray*}
\begin{aligned}
&\sum_{e\in\mathcal{E}}  A_e h_\tau^{-1}\|Q_b\jump{u_0}{}\|_e^2
\\
\leqslant&\left(\sum_{\tau\in\mathcal{T}}
h_{\tau}^2 A_{\tau}^{-1}\|f +\nabla\cdot(A_{\tau}\nabla_w \uH)\|^2_{\tau}\right)^{1/2}
\left(\sum_{\tau\in\mathcal{T}}h_{\tau}^{-2}
\|A^{1/2}_\tau(\uH-I_{\mathcal{T}} 
\uH)\|^2_{\tau}\right)^{1/2}
  \\
 &\;+ 
\left(\sum_{e\in\mathcal{E}}h_{\tau}A_e^{-1}\bigl\|\jump{A\nabla_w  
\uH\cdot \bm{n}}{}\bigr\|^2_e\right)^{1/2}
\left(\sum_{e\in\mathcal{E}} h_{\tau}^{-1} A_e\|\uH-I_{\mathcal{T}} 
\uH\|^2_e\right)^{1/2}
\\
\lesssim&\;\eta(\nabla_w \uH, \mathcal{T}) 
\cdot\left(\sum_{e\in\mathcal{E}}  h_\tau^{-1} A_e 
\|Q_b\jump{u_0}{}\|_e^2\right)^{1/2},
\end{aligned}
\end{eqnarray*}
lastly applying Lemma \ref{lem:est-nc-interp} yields the result.
\end{proof}

The proof of the upper bound mainly follows the paradigm in 
\cite{Chen;Wang;Ye:2014posteriori, Chen;Holst;Xu:2009, Alonso96:385}. Without loss of generality, 
the presentation is for $d=3$. We shall use the Helmholtz decomposition of  $\nabla u - \nabla_w \uH$.
 
\begin{theorem}[Upper Bound]
\label{thm:est-reliability}
Let $u$ be the solution of \eqref{eq:pb-model} and $\uH \in 
V_0(\mathcal{T})$ be the solution of \eqref{eq:pb-mwg}, then 
\begin{eqnarray}
\label{eq:reliability}
\|A^{1/2}(\nabla u-\nabla_{w} u_{\mathcal{T}})\|_{\mathcal{T}}^2\leqslant 
C_{U}\eta^2(\nabla_w \uH, \mathcal{T}),
\end{eqnarray} 
where the constant $C_{U}$ depends on the shape regularity of $\mathcal{T}$ and the 
ratio of the coefficient $A$ across elements.
\end{theorem}
\begin{proof} We first give an outline of our proof. The following Helmholtz decomposition
(\cite{Dari;Duran;Padra;Vampa:1996posteriori,
Carstensen;Bartels;Jansche:2002posteriori}, see also~\cite[Lemma 
4.2]{Chen;Wang;Ye:2014posteriori}, Chapter I Theorem 3.4 and Remark 
3.10 in~\cite{Girault;Raviart:1986Finite}) commonly used for nonconforming elements is applied to $\nabla u - \nabla_w \uH$:
\begin{equation}
\label{eq:decomp}
\nabla u - \nabla_w \uH = \nabla \psi + A^{-1} \nabla\times \bm{\xi},
\end{equation}
where $\psi\in H^1_0(\Omega)$, and $\bm{\xi}\in \bm{H}(\mathbf{curl}; \Omega)$. 
The 
decomposition satisfies
\begin{equation}
\label{eq:decomp-stability}
\bigl\|A^{1/2} (\nabla u- \nabla_w \uH) \bigr\|_{\mathcal{T}}^2
= \bigl\|A^{1/2}\nabla \psi \bigr\|_{\mathcal{T}}^2 + 
\bigl\|A^{-1/2}\nabla\times \bm{\xi} \bigr\|_{\mathcal{T}}^2.
\end{equation}
Moreover, $\bm{\xi}$ can be chosen to be divergence-free so that it is piecewise 
$\bm{H}^1$-smooth (e.g., see~\cite[Appendix A]{Cai;Cao:2015recovery-based}) such 
that 
\begin{equation}
\label{eq:decomp-curl}
\|A^{-1/2} \bm{\nabla} \bm{\xi}\|_{\mathcal{T}}\lesssim \|A^{-1/2} \nabla\times 
\bm{\xi}\|_{\mathcal{T}}.
\end{equation}
\begin{enumerate}
 \item For $\psi$ in \eqref{eq:decomp}, we shall construct a quasi-interpolant to $V^{\rm c}(\mathcal T)$ and use the orthogonality \eqref{eq:orthogonality-partial}. This part is similar to the \emph{a posteriori} error estimator of the conforming finite element which can be controlled by the element-wise residual and normal jump of the numerical flux. 

 \item For $\bm \xi$ in \eqref{eq:decomp}, we shall construct a similar quasi-interpolant, but this time leads to the jump of tangential derivatives. Comparing to results for WG~\cite{Chen;Wang;Ye:2014posteriori}, the element-wise $\nabla_h \times \nabla_w u_{\mathcal T}$ will be bounded by the stabilization term.
 \end{enumerate}

A key result (Lemma \ref{lem:est-jump}) is applied here to bound the stabilization term by the element-wise residual and the normal jump of the flux, where the quasi-interpolant 
to $V^{\rm nc}(\mathcal T)$ is used. As a result, we turn our focus onto the error without the stabilization as follows
\begin{equation}
\label{eq:est-I}
\bigl\|A^{1/2} (\nabla u- \nabla_w \uH) \bigr\|_{\mathcal{T}}^2
=\underbrace{\bigl(A(\nabla u- \nabla_w \uH), 
\nabla \psi\bigr)_{\mathcal{T}}}_{=:\mathfrak{I}_C}+
\underbrace{\bigl(\nabla u- \nabla_w \uH, 
\nabla\times \bm{\xi}\bigr)_{\mathcal{T}}}_{=:\mathfrak{I}_N}.
\end{equation}
Let $\psi_{\mathcal{T}}\in V^{\rm c}(\mathcal{T})$ be the robust Cl\'{e}ment-type 
quasi-interpolation (e.g., see 
\cite{Petzoldt:2002posteriori}) for $\psi$ such that 
the following estimates holds, 
\begin{equation}
\label{eq:est-interp-H1}
\begin{gathered}
A_{\tau}\left\|\psi-\psi_{\mathcal{T}}\right\|_{\tau}^2 
\lesssim 
h_{\tau}^2\big\|A^{1/2} \nabla \psi\big\|_{\omega(\tau)}^2,
\\
\text{ and }\quad A_{e}^{\max}\left\|\psi-\psi_{\mathcal{T}}\right\|_{e}^2 \lesssim
h_{\tau}\big\|A^{1/2} \nabla \psi\big\|_{\omega({e})}^2.
\end{gathered}
\end{equation}
Using the partial orthogonality \eqref{eq:orthogonality-partial}, integrating by 
parts for $\mathfrak{I}_C$, Cauchy-Schwarz inequality, and estimates 
in \eqref{eq:est-interp-H1}, we have
\begin{eqnarray*}
\begin{aligned}
\mathfrak{I}_C &= \bigl(A(\nabla u- \nabla_w \uH), 
\nabla (\psi - \psi_{\mathcal{T}})\bigr)_{\mathcal{T}}
\\
&=-\big( \nabla_h\cdot A(\nabla u- \nabla_w \uH), 
\psi-\psi_{\mathcal{T}}  \big)_{\mathcal{T}} + 
\sum_{\tau\in\mathcal{T}}\langle(A (\nabla u - \nabla_w \uH)\cdot\boldsymbol{n}, 
\psi-\psi_{\mathcal{T}} \rangle_{\partial \tau}
\\
&=(f +\nabla_h\cdot (A\nabla_w \uH), \psi - \psi_{\mathcal{T}})_{\mathcal{T}} 
- \sum_{e\in\mathcal{E}}
\left\langle\jump{(A\nabla_w \uH) \cdot\bm{n}}{}, \psi - \psi_{\mathcal{T}} 
\right\rangle_e
\\
&\leq \left(\sum_{\tau\in\mathcal{T}}
h_{\tau}^2 A_{\tau}^{-1}\|f +\nabla\cdot(A\nabla_w \uH)\|^2_{\tau}\right)^{1/2}
\left(\sum_{\tau\in\mathcal{T}} h_{\tau}^{-2}A_{\tau}\|\psi - 
\psi_{\mathcal{T}}\|^2_{\tau}\right)^{1/2}
  \\
 &\;+ 
\left(\sum_{e\in\mathcal{E}}  A_e^{-1}h_\tau \bigl\|\jump{A\nabla_w  
\uH\cdot \bm{n}}{}\bigr\|^2_e\right)^{1/2}
\left(\sum_{e\in\mathcal{E}}  A_e h_\tau^{-1}
\|\psi - \psi_{\mathcal{T}}\|^2_e\right)^{1/2}
\\
&\lesssim\;\eta(\nabla_w \uH, \mathcal{T}) \|A^{1/2} \nabla \psi\big\|_{\mathcal{T}}
\leq \eta(\nabla_w \uH, \mathcal{T}) \bigl\|A^{1/2} (\nabla u- \nabla_w \uH) 
\bigr\|_{\mathcal{T}},
\end{aligned}
\end{eqnarray*}
where the constant depends on the shape-regularity of $\mathcal{T}$.

For $\mathfrak{I}_N$ in \eqref{eq:est-I}, we need a robust Cl\'{e}ment-type 
interpolation $\bm{\xi}_{\mathcal{T}}\in \bm{H}(\mathbf{curl}; \Omega)$ (see 
\cite[Theorem 4.6]{Cai;Cao:2015recovery-based}) satisfying:
\begin{equation}
\label{eq:est-interp-Hcurl}
\begin{gathered}
A_{\tau}^{-1}\left\|\bm{\xi}-\bm{\xi}_{\mathcal{T}}\right\|_{\tau}^2
\lesssim 
\sum_{\tau \subset \omega({\tau})}h_{\tau}^2\big\|A^{-1/2} \bm{\nabla} 
\bm{\xi}\big\|_{\tau}^2,
\\
\text{and} 
\quad(A_{e}^{\min})^{-1}\left\|\bm{\xi}-\bm{\xi}_{\mathcal{T}}\right\|_{e}^2 \lesssim
\sum_{\tau \subset \omega({e})}
h_{\tau}^2 \big\|A^{1/2} \bm{\nabla} \bm{\xi}\big\|_{\tau}^2.
\end{gathered}
\end{equation}
It is straightforward to verify that, by the definition of the modified 
weak gradient \eqref{eq:mwg} and the fact that $\nabla\times 
\bm{\xi}_{\mathcal{T}}\in \bm{H}(\mathrm{div}; \Omega)$, we have
\[
\begin{aligned}
\bigl(\nabla_w  \uH, 
\nabla\times \bm{\xi}_{\mathcal{T}}\bigr)_{\mathcal{T}}
&= -\bigl(u_0, 
\nabla_h\cdot(\nabla\times \bm{\xi}_{\mathcal{T}})\bigr)_{\mathcal{T}}
+\sum_{\tau \in \mathcal{T}} \langle
Q_b\{u_0\}, \nabla\times \bm{\xi}_{\mathcal{T}}\cdot \bm{n}
\rangle_{\partial \tau}=0.
\end{aligned}
\]
Consequently, as $\nabla u\in \bm{H}(\mathbf{curl}; \Omega)$ and $\bm{\xi}_{\mathcal{T}}$ can 
be inserted to $\mathfrak{I}_N$  obtain the following:
\begin{equation}
\label{eq:curl-gradwu}
\mathfrak{I}_N 
=-\bigl(\nabla_w \uH, 
\nabla\times \bm{\xi} \bigr)_{\mathcal{T}}
= -\bigl(\nabla_w \uH, 
\nabla\times (\bm{\xi}-\bm{\xi}_{\mathcal{T}})\bigr)_{\mathcal{T}}.
\end{equation}
Upon an integration by parts, $\nabla_h\times (\nabla_w \uH)$ will appear in each element, in general, this is not zero as $\nabla_w \uH$ is computed in the sense of 
 distribution. To handle this term, simply notice that $\nabla_h\times (\nabla_h u_0) = 0$ element-wise, by Lemmas 
\ref{lem:est-diff-weak} and  \ref{lem:est-jump}, and an inverse inequality
\begin{eqnarray}\nonumber
\sum_{\tau \in \mathcal{T}}h_{\tau}^2 \|A^{1/2}\nabla\times (\nabla_w \uH)\|_{\tau}^2
&=&
\sum_{\tau \in \mathcal{T}}h_{\tau}^2 \|A^{1/2}\nabla\times (\nabla_h u_0 - 
\nabla_w \uH)\|_{\tau}^2
\\ \nonumber
&\lesssim& \|A^{1/2}(\nabla_h u_0 - \nabla_w \uH)\|_{\mathcal{T}} 
\\  \label{eq:est-curl-gradwu}
&\lesssim& \eta(\nabla_w \uH, \mathcal{T}).  
 \end{eqnarray}
Integrating by parts on \eqref{eq:curl-gradwu}, applying \eqref{eq:est-curl-gradwu}, 
we have
\[
\begin{aligned}
&\bigl(\nabla_w \uH, 
\nabla\times (\bm{\xi}-\bm{\xi}_{\mathcal{T}})\bigr)_{\mathcal{T}}
\\
=& \; \bigl(\nabla\times (\nabla_w \uH), 
\bm{\xi}-\bm{\xi}_{\mathcal{T}} \bigr)_{\mathcal{T}}
- \sum_{e\in \mathcal{E}} \langle \jump{\bm{n}\times \nabla_w \uH}{},
\bm{\xi}-\bm{\xi}_{\mathcal{T}}  \rangle_e
\\
\leq & \; \left(\sum_{\tau\in\mathcal{T}}
h_{\tau}^2 A_{\tau}\|\nabla\times (\nabla_w \uH)\|^2_{\tau}\right)^{1/2}
\left(\sum_{\tau\in\mathcal{T}} h_{\tau}^{-2}A_{\tau}^{-1}
\|\bm{\xi}-\bm{\xi}_{\mathcal{T}}\|^2_{\tau}\right)^{1/2}
  \\
 &\quad+ 
\left(\sum_{e\in\mathcal{E}}  h_\tau A_e^{\min} 
\bigl\|\jump{\nabla_w \uH\times \bm{n}}{}\bigr\|^2_e\right)^{1/2}
\left(\sum_{e\in\mathcal{E}}   h_\tau^{-1}(A_e^{\min})^{-1}
\|\bm{\xi}-\bm{\xi}_{\mathcal{T}}\|^2_e\right)^{1/2}
\\
\lesssim &\; \eta(\nabla_w \uH, \mathcal{T}) \|A^{-1/2} \bm{\nabla} \bm{\xi}\|_{\mathcal{T}}.
\end{aligned}
\]
Finally by \eqref{eq:decomp-curl}, the reliability \eqref{eq:reliability} follows.
\end{proof}

By Lemma \ref{lem:est-jump} and Theorem \ref{thm:est-reliability}, the upper bound of the 
error \eqref{eq:error} follows.
\begin{corollary}
Let $u$ be the solution of \eqref{eq:pb-model} and $\uH \in 
V_0(\mathcal{T})$ be the solution of \eqref{eq:pb-mwg}, then the 
error \eqref{eq:error}
\begin{eqnarray}
\tnorm{u - \uH }_{\mathcal{T}}^2\lesssim\eta^2(\nabla_w \uH, \mathcal{T}),
\end{eqnarray} 
where the constant depends on the shape regularity of $\mathcal{T}$ and the 
ratio of the coefficient $A$ across neighboring elements.
\end{corollary}

\subsection{Efficiency}

The standard bubble function technique is opted (see 
\cite{Verfurth:2013}) to 
derive the efficiency bound, while the tangential jump part's proof follows a 
standard argument of the \emph{a posteriori} error estimation for standard WG 
discretization in~\cite{Chen;Wang;Ye:2014posteriori}. As the proofs are standard, 
we only present the results here.

For $v_{\mathcal{T}}\in V(\mathcal{T})$ and $\tau\in \mathcal{T}$, the oscillation is defined to be
\begin{equation}\label{eq:osc1}
\mathrm{osc}_{\mathcal{T}}^2(v_{\mathcal{T}}, \tau):=h_{\tau}^2\|(Id-Q_m) R(\nabla_w v_{\mathcal{T}})\|^2_{\tau},
\end{equation}
where $Q_m$ denotes the $L^2$ projection onto the set of either the piecewise 
$P_m(\tau)$ on $\tau\in\mathcal{T}$, where 
$m=\ell-2$ if $\ell\geq 2$ and $m=0$ when $\ell=1$.

For any subset $\mathcal{M}\subseteq \mathcal{T}$, we define
\begin{equation}\label{eq:osc}
\mathrm{osc}_{\mathcal{T}}^2(v_{\mathcal{T}}, \mathcal{M}):=\sum_{\tau\in 
\mathcal{M}}\mathrm{osc}_{\mathcal{T}}^2(v_{\mathcal{T}}, \tau).
\end{equation}
Note that, on $\tau\in \mathcal{T}$, using the properties of the $L^2$-projection,  
the oscillation are dominated by the estimator; namely,
\begin{equation}\label{eq:osc-eta}
\mathrm{osc}_{\mathcal{T}}^2(v_{\mathcal{T}}, \tau)\leqslant \eta_{\mathcal{T}}^2(v_{\mathcal{T}}, \tau), \ 
\mbox{for}\ \ v_{\mathcal{T}}\in V(\mathcal{T}).
\end{equation}
\begin{theorem}(Local lower bound)
\label{thm:est-efficiency} 
For all $\tau\in \mathcal{T}$, there holds
\begin{eqnarray}\label{eq:LB}
C_{L}\eta(\nabla_w \uH, \tau) \leqslant 
\|A^{1/2}(\nabla u-\nabla_w \uH)\|_{\mathcal{T}} + 
\operatorname{osc}_{\mathcal{T}}^2(\uH, \mathcal{T})
\end{eqnarray}
where the constant $C_{L}$ only depends on the shape regular of $\mathcal{T}$.
\end{theorem}

\section{An Adaptive Modified Weak Galerkin Method}
\label{sec:convergence}
In this section, first we introduce an adaptive modified weak Galerkin method (AmWG). Next, a quasi-orthogonality is proved and is further exploited to derive the convergence of AmWG. At last, we shall present the discrete reliability and propose the quasi-optimality of the AmWG. 

Henceforth, the polynomial degree is chosen to be $\ell=1$, the reason for this is that we shall borrow some classical results for Crouzeix-Raviart element to establish a penalty parameter-free convergence. Notice that the penalty parameters in schemes based on discontinuous approximation spaces are indispensable not only for the coercivity (Remark \ref{remark:ipdg}), but also for the convergence of adaptive algorithms due to the lack of a direct orthogonality result (see e.g.,~\cite{Bonito;Nochetto:2010Quasi-optimal}). For a similar nonconforming method~\cite{Owens:2014Quasi-Optimal}, one still needs to choose a sufficiently large penalty parameter to prove the convergence. By bridging the connections between the lowest order WG method and Crouzeix-Raviart element, we are able to show the convergence without the presence of a sufficiently large penalty.

\subsection{Algorithm}\label{alg:AmWG}
\begin{algorithm}[htb]
\caption{An adaptive modified weak Galerkin finite element (AmWG) cycle $[u_J, 
\mathcal{T}_J]  = \mathrm{AmWG}  (\mathcal{T}_0, f, 
\mathrm{tol},\theta)$.}
\label{alg:amwg}
\begin{algorithmic}[1]
\Require $\mathcal{T}_0$, $f$, $\texttt{tol}$, $\theta \in (0, 1)$.
\Ensure $\mathcal{T}_J$, $u_J$.

\State{$\eta = 1, k = 0.$}

\While{$\eta\geqslant \texttt{tol}$}

\State{\textbf{SOLVE}:}
Solve \eqref{eq:pb-mwg} on $\mathcal{T}_k$ to 
get the solution $u_k$;

\State{\textbf{ESTIMATE}:}  Compute $\eta = \eta(\nabla_w u_k, \mathcal{T}_k 
)$;

\State{\textbf{MARK}:  } Seek a minimum cardinality $\mathcal{M}_k \subseteq 
\mathcal{T}_k$ such that
\begin{equation}
  \label{eq:mark}
\eta^2(\nabla_w u_k, \mathcal{M}_k )\geq\theta\eta^2(\nabla_w u_k, \mathcal{T}_k );
\end{equation}

\State{\textbf{REFINE}:  } Bisect/quadsect elements in 
$\mathcal{M}_k$ and the neighboring elements to form a conforming 
$\mathcal{T}_{k+1}$;

\State{$k\gets k+1$}
\EndWhile

\State{$u_J = u_k; \mathcal{T}_J = \mathcal{T}_k$.}

\end{algorithmic}
\end{algorithm} 

In the {\bf SOLVE} step, given a function ${f}\in {L}^2(\Omega)$ 
and a triangulation $\mathcal{T}$, the exact discrete solution is sought 
$u_\mathcal{T} = \mathbf{SOLVE}(\mathcal{T}, f)$. In this 
step, we assume that the discrete linear system associated with problem  
\eqref{eq:pb-mwg} can be solved exactly.

In the {\bf ESTIMATE} step, local error indicators $\{\eta(\nabla_w \uH, \tau)\}_{\tau\in\mathcal{T}}$ 
and the global estimator $\eta(\nabla_w \uH,  \mathcal{T})$ are calculated. 

In the {\bf MARK} step, a set of marked elements is obtained by the D\"orfler 
marking strategy~\cite{Dorfler:1996convergent} applied on the error indicators 
$\{\eta(\nabla_w \uH, \tau)\}_{\tau\in\mathcal{T}}$ on $\mathcal{T}$ obtained in the {\bf ESTIMATE} step. 

In the {\bf REFINE} step, different from traditional DG approaches which allow 
hanging nodes (e.g.,~\cite{Bonito;Nochetto:2010Quasi-optimal}), the marked 
elements, as well as their neighbors, are 
refined using bisection ($d=2,3$) or red-green refinement ($d=2$) while preserving the 
\emph{conformity} of the triangulation.  

In the paragraphs hereafter, the notation $\mathcal{T}_1 \leqslant \mathcal{T}_2$ 
stands for that $\mathcal{T}_2$ is a refinement of $\mathcal{T}_1$ following the 
marking strategy above, where $\mathcal{T}_1, \mathcal{T}_2\in 
\mathcal{C}(\mathcal{T}_0)$, 
and here $\mathcal{C}(\mathcal{T}_0)$ 
denotes the set of triangulations which are conforming (no hanging nodes), shape 
regular and refined from an initial triangulation $\mathcal{T}_0$. 

While showing the lemmas related to the convergence of the 
AmWG, for $\mathcal{T}, 
\mathcal{T}_*\in \mathcal{C}(\mathcal{T}_0)$ and 
$\mathcal{T}\leqslant \mathcal{T}_*$, the set of refined elements in 
$\mathcal{T}$, which become new elements in $\mathcal{T}_*$, is denoted as 
\[
\mathcal{R}_{\mathcal{T}\rightarrow\mathcal{T}_*}:=\{\tau\in \mathcal{T}: 
\;\tau\not\in \mathcal{T}_*\}\subset \mathcal{T}.
\]

Whenever the dependence of the weak gradient on two different meshes becomes relevant, the weak gradient's notation is changed accordingly to emphasize the mesh of the function defined on, for example, the piecewisely defined weak gradient is $\nabla_{w,*} v$ for any $v\in \mathcal{T}_*$. When its restriction to one element $\tau\in \mathcal T$ is of interest, $\nabla_{w,\tau} v_{\mathcal T}$ is used.

\begin{remark}
\label{remark:eta-gradw}
The reason why we opt for a notation $\eta(\nabla_w \vH, \mathcal{M})$, not $\eta(\vH, \mathcal{M})$ is as follows. As in the context of the convergence analysis, this chosen notation has a more consistent meaning when considering two meshes: one is refined from the other.
Note that for two nested triangulation $\mathcal{T}\leqslant \mathcal{T}_*$, the weak gradient of a coarse function $v_{\mathcal T}\in V(\mathcal T)$ on the fine mesh $\mathcal T_*$ is different with the weak gradient on the coarse grid $\mathcal T$. 
To be specific, on $\tau_*\subset \tau\in \mathcal{T}$ with $\tau_*\not \in \mathcal{T}$, $\nabla_{w,\tau_*} \vH$ is different from $(\nabla_{w,\tau} \vH)|_{\tau_*}$. We note that this is different from that of the piecewise gradient $\nabla_h$.
\end{remark}

\subsection{Reduction of error estimator}\label{sec:est-eta}
By $\ell=1$, the error estimator defined in \eqref{eq:eta} can be split as
\begin{equation}
\label{eq:eta2}
\eta^2(\nabla_w\vH, \tau)
:= F(f, \tau) + \tilde{\eta}^2(\nabla_w\vH, \tau)
\end{equation}
where 
\begin{eqnarray*}
\begin{aligned}
F(f, \tau):=h_\tau^2A_\tau^{-1}\|f\|_{0, \tau}^2;
\end{aligned}
\end{eqnarray*}
and
\begin{eqnarray*}
\begin{aligned}
\tilde{\eta}^2(\nabla_w\vH, \tau):=\sum_{e\subset\partial\tau}h_{\tau}\int_{e}
 \left(\big(A_e^{\max}\big)^{-1} J^2_{n,e}(A\nabla_w \vH) 
+ A_e^{\min} J^2_{t,e}(\nabla_w \vH)\right)\,\mathrm{d}s.
\end{aligned}
\end{eqnarray*}
For any subset $\mathcal{M}\subset \mathcal{T}$, define
\begin{eqnarray*}
F(f, \mathcal{M})=\sum_{\tau\in\mathcal{M}} F(f, \tau);\quad
\tilde{\eta}^2(\vH, \mathcal M)=\sum_{\tau\in\mathcal{M}}\tilde{\eta}^2(\vH, \tau).
\end{eqnarray*}

The next lemma shows the reduction of the error estimator after the mesh is refined.
On a refined mesh, the effect of changing the finite element function for $\tilde{\eta}^2(\vH, \mathcal{T})$ is as follows.
\begin{lemma}
\label{lem:est-eta-1}
For $\mathcal{T}, \mathcal{T}_*\in \mathcal{C}(\mathcal{T}_0)$ with 
$\mathcal{T}\leqslant \mathcal{T}_*$, for any $\zeta\in(0, 1)$, $v\in V(\mathcal{T})$, and $v_*\in V(\mathcal{T}_*)$, there exists a constant $C_{E}$ 
depending on the shape regularity $\mathcal{T}_{*}$ such that 
\begin{equation} 
\begin{aligned}\label{eq:etaestimate1}
\lefteqn{\tilde{\eta}^2(\nabla_{w, *}v_{*}, \mathcal{T}_*)}\\ 
&\leqslant
(1+\zeta)\tilde{\eta}^2(\nabla_{w,\tau}\vH, \mathcal{T}_{*}) 
 +C_{E}(1+\zeta^{-1})\|A^{1/2}(\nabla_{w, *}v_{*}-\nabla_{w, 
\tau}\vH)\|^2_{\mathcal{T}_{*}}.
\end{aligned}
\end{equation}
\end{lemma}

We can also get the contraction of $\tilde{\eta}^2(\nabla_{w,\tau}\uH, \mathcal{T})$ if the weak flux of the solution remains invariant and is interpolated into a finer mesh refined using the D\"orfler marking strategy. 
\begin{lemma}
\label{lem:est-eta-2}
For $\mathcal{T}, \mathcal{T}_*\in \mathcal{C}(\mathcal{T}_0)$ with 
$\mathcal{T}\leqslant \mathcal{T}_*$. Let $\uH\in V_0(\mathcal{T})$ be 
the solution to \eqref{eq:pb-mwg}. For $\lambda\in(0, 1)$ defined in
 Lemma \ref{lem:estimate-f}, we have
\begin{eqnarray}
\label{eq:etaestimate2}
\tilde{\eta}^2(\nabla_{w,\tau}\uH, \mathcal{T}_{*})\leqslant
\tilde{\eta}^2(\nabla_{w,\tau}\uH, \mathcal{T})
-\lambda\tilde{\eta}^2(\nabla_{w,\tau}\uH, \mathcal{R}_{\mathcal{T}\rightarrow\mathcal{T}_*}).
\end{eqnarray}
\end{lemma}
Here we skip the proof of Lemmas \ref{lem:est-eta-1} and   \ref{lem:est-eta-2}, since the corresponding techniques are quite standard and can be found, e.g. in~\cite{CasconKreuzer08:2524}.

Hereafter the following short notations are adopted: on $\mathcal{T}_k$, $\nabla_{w, k}$ denotes the weak gradient, and $\nabla_k$ denotes the piecewise gradient $\nabla_h$, for quantities involving two levels of meshes, the subscript follows that of the coarse one.
\begin{eqnarray*}
& \varepsilon_{k}=\|A^{1/2}(\nabla u-\nabla_{w, k}u_{\mathcal{T}_k})\|^2;
\\
& E_k=\|A^{1/2}(\nabla_{w, k}u_{\mathcal{T}_k}-\nabla_{w, k+1}u_{\mathcal{T}_{k+1}})\|^2;
\\
& R_k=\mathcal{R}_{\mathcal{T}_k\rightarrow\mathcal{T}_{k+1}};
\\
& \eta_k=\eta(\nabla_{w, k} u_{\mathcal{T}_{k}}, \mathcal{T}_{k});\quad 
\eta_{R_k}=\eta(\nabla_{w, k} u_{\mathcal{T}_{k}}, \mathcal{R}_{\mathcal{T}_k\rightarrow\mathcal{T}_{k+1}});
\\
&\tilde{\eta}_k=\tilde{\eta}(\nabla_{w, k} u_{\mathcal{T}_{k}}, \mathcal{T}_{k});\quad
\tilde{\eta}_{R_k}=\tilde{\eta}(\nabla_{w, k} u_{\mathcal{T}_{k}}, 
\mathcal{R}_{\mathcal{T}_k\rightarrow\mathcal{T}_{k+1}});
\\
&F_k=F(f, \mathcal{T}_k);\quad
F_{R_k}=F(f, \mathcal{R}_{\mathcal{T}_k\rightarrow\mathcal{T}_{k+1}}).
\end{eqnarray*}

The following lemma summarizes the contraction of $\tilde{\eta}^2(\cdot, 
\cdot)$ by using Lemma \ref{lem:est-eta-1} and Lemma \ref{lem:est-eta-2}. 
\begin{lemma}
\label{lem:eta1} 
For the two consecutive triangulation $\mathcal{T}_{k}\leq \mathcal{T}_{k+1}$ in the 
AmWG cycle \eqref{alg:amwg}, and any $\zeta\in(0, 1)$, there exists a constant $\rho>0$ depending on the shape regularity $\mathcal{T}_{k+1}$  
 such that
\begin{equation}
\label{eq:est-eta1}
\tilde{\eta}_{k+1}^2
\leqslant
(1+\zeta)(\tilde{\eta}_k^2-\lambda\tilde{\eta}_{R_k}^2) + E_k/\rho.
\end{equation}
\end{lemma}
\begin{proof}
Let $\mathcal{T} = \mathcal{T}_{k}$ and $\mathcal{T}_{*} = \mathcal{T}_{k+1}$ in 
Lemmas \ref{lem:est-eta-1} and  \ref{lem:est-eta-2}, respectively. Then the desired 
result \eqref{eq:est-eta1} follows from letting $\rho^{-1}= C_{E}(1+\zeta^{-1})$.
\end{proof}

For the contraction of $F(\cdot, \cdot)$, there is no extra $\zeta$ factor as an artifact of Young's inequality, and this will play a key role in proving the convergence without any penalty parameter on 
the stabilization.
\begin{lemma}\label{lem:estimate-f}
Let $\mathcal{T}_{k+1}$ be the refinement of $\mathcal{T}_{k}$ produced in Algorithm \ref{alg:amwg}. There exists a constant $\lambda\in(0, 1)$ satisfying
\begin{equation}\label{eq:estimate-f}
F_{k+1}\leqslant F_k - \lambda F_{R_k}.
\end{equation}
\end{lemma}
\begin{proof}
For any $\tau\in \mathcal{T}\backslash\mathcal{T}_*$, we only 
need to consider the case where $\tau$ is subdivided into 
$\tau_*^1, \tau_*^2\in \mathcal{T}_*$ with $h_{\tau_*^1}^{d}
=h_{\tau_*^{2}}^{d}=\frac{1}{2}h_{\tau_*}^{d} (d=2, 3)$, we have
\begin{eqnarray*}
F_{k+1}&=&\sum_{\tau \in \mathcal{T}_{k+1}} h_{k+1}^2A_{k+1}^{-1}\|f\|_{0, \tau}^2
\\
& =& \sum_{\tau \in \mathcal{T}_{k+1}\cap\mathcal{T}_{k}} h_{k+1}^2A_{k+1}^{-1}\|f\|_{0, \tau}^2
+ \sum_{\tau \in \mathcal{T}_{k+1}\backslash\mathcal{T}_k} h_{k+1}^2A_{k+1}^{-1}\|f\|_{0, \tau}^2\\
&\leqslant& \sum_{\tau \in \mathcal{T}_k} h_{k}^2A_{k}^{-1}\|f\|_{0, \tau}^2 
-\sum_{\tau \in \mathcal{T}_k\backslash \mathcal{T}_{k+1}} h_{k}^2A_{k}^{-1}\|f\|_{0, \tau}^2\\
&&\quad +\sum_{\tau \in \mathcal{T}_k\backslash \mathcal{T}_{k+1}} 2^{-1/d} h_{k}^2A_{k}^{-1}\|f\|_{0, \tau}^2\\
&=& F_k - \lambda F_{R_k},
\end{eqnarray*}
where $\lambda:=1-2^{-1/d}\in (0, 1)$.
\end{proof} 

The following lemma summarizes the contraction of error estimator $\eta^2(\cdot, 
\cdot)$ by using Lemmas \ref{lem:eta1} and  \ref{lem:estimate-f}. 

\begin{lemma}\label{lem:eta}
There exist constants $\zeta\in(0, 1)$ and $\rho > 0$ such that
\begin{equation}
\label{eq:est-eta}
\eta^2_{k+1}
\leqslant 
(1+\zeta)(1-\theta\lambda) \eta^2_{k}-\zeta F_k + \zeta\lambda F_{R_k}
+E_k/\rho,
\end{equation}
where the parameters $\theta$ and $\lambda$ are given in the marking strategy \eqref{eq:mark}, Lemmas \ref{lem:estimate-f} and   \ref{lem:eta1}, respectively.
\end{lemma}

\begin{proof}
Making use of the simplified notation of $\eta(\cdot, \cdot)$ in \eqref{eq:eta2}, Lemmas \ref{lem:estimate-f} and  \ref{lem:eta1}, we have
\begin{eqnarray}\nonumber
\eta^2_{k+1}
&=&
\tilde{\eta}^2_{k+1}+ F_{k+1}
\\ \nonumber
&\leqslant& (1+\zeta)\left(\tilde{\eta}^2_k -\lambda\tilde{\eta}^2_{R_k}\right)
+  F_k-F_{R_k}+ E_k/\rho
\\ \label{eq:eta1} 
&\leqslant &(1+\zeta)\left(\eta^2_k -\lambda\eta^2_{R_k}\right)
- \zeta( F_k-\lambda F_{R_k})+ E_k/\rho.
\end{eqnarray}

By \eqref{eq:mark} and $\mathcal{T}_{k+1}$ being refined at least 
once from $\mathcal{T}_k$, we 
have $\eta^2_{R_k}\geqslant \theta\eta^2_k$. In conjunction with \eqref{eq:eta1}, we obtain
\begin{eqnarray*}
\eta^2_{k+1}
\leqslant 
(1+\zeta)(1-\theta\lambda)\eta^2_{k} - \zeta F_k + \zeta\lambda F_{R_k}+ E_k/\rho,
\end{eqnarray*}
which completes the proof.
\end{proof}

\subsection{Quasi-orthogonality}
In this section,  we will show the contraction property of the energy error by using similar arguments in~\cite{Hu;Xu;2013}. 

First, a canonical interpolation operator $I_{\mathcal{T}}^{^{\rm CR}}$ is defined any $v\in H_0^1(\Omega)$: $I_{\mathcal{T}}^{^{\rm CR}} v\in V^{\mathrm{nc}}(\mathcal{T})$ satisfies 
\begin{equation}\label{eq:interpolation-pi}
\int_e I_{\mathcal{T}}^{^{\rm CR}} v =\int_e v, \quad \forall e\in \mathcal{E},
\end{equation}
and the interpolation admits the following estimate:
\begin{equation}\label{eq:I-interpolation1}
\big\|v-I_{\mathcal{T}}^{^{\rm CR}} v\big\|_{\tau} \lesssim h_\tau\|\nabla v\|_{\tau}, \quad \forall \tau\in \mathcal{T},
\end{equation}
where the constant depends only on the shape regularity of $\tau$.
\begin{lemma}\label{lem:I-projection1}
Assume that $\bm{\sigma}_{\mathcal{T}}$ is a constant vector on each element $\tau\in \mathcal{T}$, we have
\begin{equation}
(\nabla_h v, \bm{\sigma}_{\mathcal{T}})_{\mathcal{T}} =(\nabla_h I_{\mathcal{T}}^{^{\rm CR}} v, \bm{\sigma}_{\mathcal{T}})_{\mathcal{T}}.
\end{equation}
\end{lemma}
\begin{proof}
Note that $\bm{\sigma}_{\mathcal{T}}$ is a constant vector on each $\tau\in \mathcal{T}$, $\nabla\cdot (\bm{\sigma}_{\mathcal{T}}|_\tau)=0$. By applying integration by parts, we will get the desired result.
\end{proof}

For $v_0\in V^{\scriptscriptstyle{\rm DG}}(\mathcal{T})$, it is embedded into $V(\mathcal{T})$ by $(v_0, Q_b\{v_0\})$. Denote the interpolant $I_{\mathcal{T}}^{^{\rm CR}}$ for $V^{\mathrm{nc}}(\mathcal{T})$ satisfying 
\begin{equation}
\label{eq:I-interpolation2}
\int_e I_{\mathcal{T}}^{^{\rm CR}} v_0 =\int_e Q_b\{v_0\}, \quad \forall e\in \mathcal{E}.
\end{equation}
\begin{lemma}
\label{lem:I-projection2}
For any $v_0\in V^{\scriptscriptstyle{\rm DG}}(\mathcal{T})$ and $\vH=\{v_0, Q_b\{v_0\}\}\in V(\mathcal{T})$, the interpolation defined in \eqref{eq:I-interpolation2} satisfies
\begin{equation}\label{eq:I-projection2}
\nabla_{w, \tau} \vH=\nabla_h I_{\mathcal{T}}^{^{\rm CR}} v_0.
\end{equation}
\end{lemma}
\begin{proof}
 For any constant vector $\bm{\sigma}_{\mathcal{T}}$, using the definition \eqref{Def:MWG} leads to 
\begin{eqnarray*}
(\nabla_{w, \tau} \vH, \bm{\sigma}_{\mathcal{T}})_\tau
&=&(v_0, \nabla\cdot \bm{\sigma}_{\mathcal{T}})+\langle Q_b \{v_0\}, \bm{\sigma}_{\mathcal{T}}\cdot \bm{n}\rangle_{\partial \tau}
\\
&=&\langle I_{\mathcal{T}}^{^{\rm CR}} \vH, \bm{\sigma}_{\mathcal{T}}\cdot \bm{n}\rangle_{\partial \tau}
\\
&=&(\nabla_{h} I_{\mathcal{T}}^{^{\rm CR}} v_0, \bm{\sigma}_{\mathcal{T}})_\tau.
\end{eqnarray*}

As $\nabla_{w, \tau} \vH, \nabla_{h} I_{\mathcal{T}}^{^{\rm CR}} v_0$ are constant on each element $\tau\in \mathcal{T}$, we get \eqref{eq:I-projection2}.
\end{proof}

\begin{lemma}[Quasi-orthogonality]
\label{lem:quasi-orthogonality}
Let $u\in H_0^1(\Omega)$ is the weak solution of \eqref{eq:pb-model}. Let $\mathcal{T}_{k+1}$ be the refinement of  $\mathcal{T}_{k}$ produced in Algorithm \ref{alg:amwg},  
$u_k=\{u_0^k, u_b^k\}\in V_0(\mathcal{T}_k)$ and
$u_{k+1}=\{u_0^{k+1}, u_b^{k+1}\} \in V_0(\mathcal{T}_{k+1})$ be the solutions to mWG discretization 
\eqref{eq:pb-mwg}, respectively. For any positive constant $\delta\in(0, 1)$, we have
\begin{equation}
\label{eq:quasi-orthogonality}
(1-\delta)\varepsilon_{k+1} \leqslant \varepsilon_{k} -E_k+ ( 
C_1 F_{R_k}) /\delta,
\end{equation}
with constant $C_1$ depending on the shape regularity of $\mathcal{T}_{k+1}$.
\end{lemma}

\begin{proof}
Let $v=u-u_0^k$. Notice that $\nabla_{w,k+1}u_{k+1}-\nabla_{w,k}u_k$ is a piecewise constant on $\tau \in \mathcal{T}_{k+1}$, we have 
\begin{eqnarray*}
\lefteqn{(A(\nabla u-\nabla_{w,k+1}u_{k+1}), \nabla_{w,k+1}u_{k+1}-\nabla_{w,k}u_k)_{\mathcal{T}_{k+1}}}\\
&\ \ \ \ =(A\nabla_{k+1} I_{k+1}^{^{\rm CR}} v, \nabla_{w,k+1}u_{k+1}-\nabla_{w,k}u_k)_{\mathcal{T}_{k+1}}.
\end{eqnarray*}
As $Q_b\jumpat{I_{k+1}^{^{\rm CR}}v}{e_{k+1}} =0$, $\forall e_{k+1}\in \mathcal{E}_{k+1}$, we get 
\begin{eqnarray*}
(A\nabla_{w,k+1}u_{k+1}, \nabla_{h} I_{k+1}^{^{\rm CR}} v)_{\mathcal{T}_{k+1}}=(f, \nabla_{h} I_{k+1}^{^{\rm CR}} v),
\end{eqnarray*} 
and 
\begin{eqnarray*} 
(A\nabla_{w,k}u_k, \nabla_{h} I_{k+1}^{^{\rm CR}} v)_{\mathcal{T}_{k+1}}
&=&(A\nabla_{w,k}u_k, \nabla_{k}  I_{k}^{^{\rm CR}} ( I_{k+1}^{^{\rm CR}} v))_{\mathcal{T}_{k+1}}\\
&=&(f, \nabla_k  I_{k}^{^{\rm CR}} ( I_{k+1}^{^{\rm CR}} v))_{\mathcal{T}_{k+1}}.
\end{eqnarray*}

For any $\tau\in \mathcal{T}_k\cap\mathcal{T}_{k+1}$, we have $ I_{k}^{^{\rm CR}}v|_\tau= I_{k+1}^{^{\rm CR}}v|_\tau$, then 
\begin{eqnarray*}
(A\nabla_{k+1} I_{k+1}^{^{\rm CR}} v, \nabla_{w,k+1}u_{k+1}-\nabla_{w,k}u_k)_{\mathcal{T}_{k+1}}=0.
\end{eqnarray*}
Applying Lemma \ref{lem:I-projection1}, Lemma \ref{lem:I-projection2}, the triangle inequality, Lemma 
\ref{lem:est-diff-weak}, and Lemma \ref{lem:est-jump} leads to 
\begin{eqnarray}\nonumber
\lefteqn{(A\nabla_{k+1} I_{k+1}^{^{\rm CR}} v, \nabla_{w,k+1}u_{k+1}-\nabla_{w,k}u_k)_{\mathcal{T}_{k+1}}}
\\ \nonumber
&&=(f, (I- I_{k}^{^{\rm CR}}) I_{k+1}^{^{\rm CR}} v)_{\mathcal{T}_{k+1}}
\\ \nonumber
&&\leqslant \sum_{\tau \in \mathcal{T}_{k+1}\backslash \mathcal{T}_k}h_{k+1}\|f\|_{\tau_{k+1}} \cdot h_{k+1}^{-1}\|(I- I_{k}^{^{\rm CR}}) I_{k+1}^{^{\rm CR}} v\|_{\tau_{k+1}}
\\ \nonumber
&&\lesssim\sum_{\tau \in \mathcal{T}_k\backslash 
\mathcal{T}_{k+1}}h_{k}\|f\|_{\tau_k} \cdot \|\nabla_{k+1} I_{k+1}^{^{\rm CR}} 
v\|_{\mathcal{T}_{k+1}}
\\ \label{eq:quasi-orthogonality1}
&&\leqslant \sqrt{C_1} F_{R_k}^{1/2}\cdot \varepsilon_{k+1}.
\end{eqnarray}

It follows the Young's inequality and \eqref{eq:quasi-orthogonality1} that
\begin{eqnarray*}
\varepsilon_{k+1}
&=&\varepsilon_{k}-E_k-2(A(\nabla u-\nabla_{w,k+1}u_{k+1}), \nabla_{w,k+1}u_{k+1}-\nabla_{w,k}u_k)_{\mathcal{T}_{k+1}}\\
&\leqslant &\varepsilon_{k}-E_k+{\delta}\varepsilon_{k+1} +{\delta}\eta_{k+1} + ( 
C_1F_{R_k}) /\delta.
\end{eqnarray*}
\end{proof}

\subsection{Convergence of AmWG}
In the following theorem, the convergence of Algorithm \ref{alg:amwg} is proved. The main idea is to use the negative term on the right to cancel the positive terms, and to use the reduction factor $\lambda\in(0, 1)$ in \eqref{eq:estimate-f} and \eqref{eq:est-eta}. 

\begin{theorem}\label{thm:convergence}
Given a marking parameter $\theta\in(0, 1)$ and an initial mesh $\mathcal{T}_0$. Let $u$ be the solution of \eqref{eq:pb-model}, $\{\mathcal{T}_k, u_k, \eta(\nabla_w u_k,\mathcal{T}_k)\}_{k\geq 0}$ be a sequence of meshes, finite element approximations and error estimators produced by Algorithm \ref{alg:amwg} with $\ell=1$, then there exist constants $\sigma\in (0,1)$, $\rho>0$, $C_2>0$ depending only on the shape regularity of $\mathcal{T}_0$, the marking parameter $\theta$, and $\delta$, such that if
\begin{eqnarray*}
0<\delta<\min\left(\dfrac{\rho (1-(1+\zeta)(1-\theta\lambda))}{C_U}, 1\right),
\end{eqnarray*}
then
\begin{eqnarray*}
(1 - \delta)\varepsilon_{k+1} + \rho\eta^2_{k+1} + C_2F_{k+1}
\leqslant\sigma\left((1 - \delta)\varepsilon_{k} + \rho\eta^2_k + C_2F_k\right),
\end{eqnarray*}
where the constant $C_U$ is given by Theorem \ref{thm:est-reliability}. 
\end{theorem}
\begin{proof}
By adding $\rho \eta^2_{k+1}$ to both sides of \eqref{eq:quasi-orthogonality}, then applying Lemma \ref{lem:eta}, we have  
\begin{eqnarray}\nonumber
\lefteqn{(1 - \delta)\varepsilon_{k+1} + \rho\eta^2_{k+1}}
\\ \label{eq:convergence1}
&&\leqslant \varepsilon_{k} + \rho(1+\zeta)(1-\theta\lambda)\eta^2_{k}+ 
\left(\zeta\lambda\rho+C_1/\delta\right)F_{R_k} -\rho\zeta F_k,
\end{eqnarray}
for any constant $\delta\in(0, 1)$. Let $C_2>0$ be a to-be-determined constant, by adding $C_2F_{k+1}$ in the both sides of \eqref{eq:convergence1} and applying Lemma \ref{lem:estimate-f}, we obtain
\begin{eqnarray}\nonumber
\lefteqn{(1 - \delta)\varepsilon_{k+1} + \rho\eta^2_{k+1}+C_2F_{k+1}}
\\ \label{eq:convergence2}
&&\leqslant \varepsilon_{k}  + \rho(1+\zeta)(1-\theta\lambda)\eta^2_{k}+ 
(\zeta\lambda\rho+C_1/\delta-\lambda C_2)F_{R_k}+(C_2-\rho\zeta) F_k. \ \quad \ \ \ \
\end{eqnarray}

The inequality above \eqref{eq:convergence2} along with a sufficiently large $C_2$ satisfying 
\begin{equation}\label{eq:convergence3}
\zeta\lambda\rho+C_1/\delta-\lambda C_2\leqslant 0,
\end{equation}
combining Theorem \ref{thm:est-reliability} with \eqref{eq:convergence2}, yields
\begin{eqnarray*} 
\begin{aligned}
\lefteqn{(1 - \delta)\varepsilon_{k+1} + \rho\eta^2_{k+1}+C_2F_{k+1}}\\
&\leqslant \varepsilon_{k}  + 
\rho(1+\zeta)(1-\theta\lambda)\eta^2_{k}+(C_2-\rho\zeta) F_k\\
&\leqslant \sigma_1(1 - \delta) \varepsilon_{k} + 
\big(C_U-C_U\sigma_1(1 - \delta) + \rho(1+\zeta)(1-\theta\lambda)\big) 
\eta^2_{k}+(C_2-\rho\zeta) F_k, 
\end{aligned}
\end{eqnarray*}
according to 
\begin{eqnarray*} 
\rho\sigma_1 = C_U-C_U\sigma_1(1 - \delta) + 
\rho(1+\zeta)(1-\theta\lambda),
\end{eqnarray*}
choose
\begin{eqnarray*} 
\sigma_1 = \dfrac{C_U+\rho(1+\zeta)(1-\theta\lambda)}{C_U+\rho - 
C_U\delta}.
\end{eqnarray*}
Choosing $\zeta$ satisfies $(1+\zeta)(1-\theta\lambda)(1+\delta/\rho)\in(0, 1)$ 
and the requirement 
$0<\delta<\min\left(\dfrac{\rho(1-(1+\zeta)(1-\theta\lambda))}{C_U}, 1\right)$ 
lead to $\sigma_1\in(0, 1)$. By \eqref{eq:convergence3}, we obtain 
$C_2-\rho\zeta>0$.  Now letting $\sigma_2 = (C_2- \rho\zeta)/C_2$ results 
$\sigma_2\in(0, 1)$, and
\begin{eqnarray*}
(1 - \delta)\varepsilon_{k+1} + \rho\eta^2_{k+1}+C_2F_{k+1}
\leqslant\sigma_1(1 - \delta)\varepsilon_{k} + \rho\sigma_1\eta^2_k + C_2\sigma_2F_k.
\end{eqnarray*}
We complete the proof by setting $\sigma= \max\{\sigma_1, \sigma_2\}\in(0, 1)$.
\end{proof}

By recursion, the decay of the error plus the estimator is as follows.

\begin{corollary}
Under the hypotheses of Theorem \ref{thm:convergence}, then we have
\begin{eqnarray*} 
(1 - \delta)\varepsilon_{k+1} + \rho\eta^2_{k+1}+C_2F_{k+1}
\leqslant C_0\sigma^k,
\end{eqnarray*}
where the constant $\rho, \delta$ are given in Theorem \ref{thm:convergence}, and 
$C_{0} = (1 - \delta)e_0 + \rho\eta^2_{0} + C_2F_0$. As a result, The AmWG in Algorithm \ref{alg:amwg} will terminate in finite steps.
\end{corollary}

\subsection{Discrete Reliability}
In this section, we prove the discrete reliability. Let $\mathcal{T}_{k+1}$ be a refinement from $\mathcal{T}_{k}$, 
we recall the projection operator $J_{k+1}:V^{\mathrm{nc}}(\mathcal{T}_k)\rightarrow V^{\mathrm{nc}}(\mathcal{T}_{k+1})$ (see~\cite[Section 5]{Hu;Xu;2013}).
\begin{lemma}[\cite{Hu;Xu;2013}, Lemma 5.1]
\label{lem:j-projection}
For any $v_k\in V^{\mathrm{nc}}(\mathcal{T}_{k})$, it holds that
\begin{equation}\label{eq:j-projection}
\|\nabla_{k+1}(J_{k+1}v_k-v_k)\|_{\mathcal{T}}\lesssim 
\left(\sum_{\tau\in R_k} \sum_{e\subset \partial\tau}
  h_\tau\|J_{t,e}(\nabla_k v_k)\|_e^2\right)^{1/2}.
\end{equation}
\end{lemma}

\begin{remark}
Lemma \ref{lem:j-projection} was presented in~\cite{Hu;Xu;2013} for Stokes problem with a vector function in the Crouzeix-Raviart space, with a bound using $\|\jump{(\nabla_k \bm{v}_k)\bm{t}_e]}{}\|_e$ on an edge (2D). However, since the proof only relies on the scaling of the nodal basis function, and a partition of unity property of the basis on an edge, both of which holds in 3D tetrahedral nodal basis associated with faces, the result holds for scalar functions by choosing only 1 non-trivial component in the vectorial result, and acknowledging the fact that $\|\jump{\nabla_k v_k\times \bm{n}_e}{}\|_e = 
\|\jump{\operatorname{Proj}_e(\nabla_k v_k)}{}\|_e$ if $e$ is a face on $\partial \tau$.
\end{remark}

\begin{lemma}
\label{lem:discrete-reliability}
The following discrete reliability holds with constant $C_{dr}$ depending on the shape regularity of the mesh
\begin{equation}
E_k\leqslant C_{dr} \eta_{R_k}^2.
\end{equation}
\end{lemma}
\begin{proof}
By Lemma \ref{lem:I-projection1} and \ref{lem:I-projection2}, we obtain 
\begin{eqnarray}\nonumber
\hspace{-0.28cm} E_k&=&\|A^{1/2}(\nabla_{w, k+1}u_{k+1}- \nabla_{w, k}u_{k})\|_{\mathcal{T}_{k+1}}^2
\\ \nonumber
&=&
(A(\nabla_{w, k+1}u_{k+1}- \nabla_{w, k}u_{k}), \nabla_{k+1}I_{k+1}^{^{\mathrm{CR}}}u_0^{k+1}-\nabla_{k}I_{k}^{^{\mathrm{CR}}}u_0^{k})_{\mathcal{T}_{k+1}}
\\ \nonumber
&=& (A(\nabla_{w, k+1}u_{k+1}- \nabla_{w, k}u_{k}), \nabla_{k+1}I_{k+1}^{^{\mathrm{CR}}}u_0^{k+1}-\nabla_{k+1}J_{k+1}I_{k}^{^{\mathrm{CR}}}u_0^{k})_{\mathcal{T}_{k+1}}
\\ \label{eq:dr-1}
&&\hspace{-0.18cm} +\  (A(\nabla_{w, k+1}u_{k+1}- \nabla_{w, k}u_{k}), \nabla_{k+1}J_{k+1} I_{k}^{^{\mathrm{CR}}}u_0^{k}-\nabla_{k}I_{k}^{^{\mathrm{CR}}}u_0^{k})_{\mathcal{T}_{k+1}}.\ \ \ 
\end{eqnarray}
For the first term on the right-hand side of the equation above,  denote 
$v_{k+1}^{\scriptscriptstyle{\mathrm{CR}}}=I_{k+1}^{^{\mathrm{CR}}}u_0^{k+1}-J_{k+1}I_{k}^{^{\mathrm{CR}}}u_0^{k}$, it follows from $Q_b\jumpat{v_{k+1}^{\scriptscriptstyle{\mathrm{CR}}}}{e_{k+1}} =0$, $\forall e_{k+1}\in \mathcal{E}_{k+1}$ that
\begin{eqnarray*}
(A\nabla_{w, k+1}u_{k+1}, \nabla_{k+1}v_{k+1}^{\scriptscriptstyle{\mathrm{CR}}})_{\mathcal{T}_{k+1}}=(f, v_{k+1}^{\scriptscriptstyle{\mathrm{CR}}}),
\end{eqnarray*}
and $(\nabla_{k+1} v_{k+1}^{\scriptscriptstyle{\mathrm{CR}}}, \bm{\sigma})_{\tau} = (\nabla_k I_k^{^{\mathrm{CR}}} v_{k+1}^{\scriptscriptstyle{\mathrm{CR}}}, \bm{\sigma})_\tau$ for a constant vector $\bm{\tau}$ on $\forall \tau\in \mathcal{T}_k\backslash \mathcal{T}_{k+1}$
\begin{eqnarray*}
\begin{aligned}
(A\nabla_{w, k}u_{k}, \nabla_{k+1}v_{k+1}^{\scriptscriptstyle{\mathrm{CR}}})_{\mathcal{T}_{k+1}}
=(A\nabla_{w, k}u_{k}, \nabla_{k}I_{k}^{^{\mathrm{CR}}}v_{k+1}^{\scriptscriptstyle{\mathrm{CR}}})_{\mathcal{T}_{k}}
=(f, I_{k}^{^{\mathrm{CR}}}v_{k+1}^{\scriptscriptstyle{\mathrm{CR}}}).
\end{aligned}
\end{eqnarray*}
Combining both further implies 
\begin{eqnarray} \nonumber
\lefteqn{
\big(A(\nabla_{w, k+1}u_{k+1}-\nabla_{w, k}u_{k}), \nabla_{k}v_{k}^{\scriptscriptstyle{\mathrm{CR}}}\big)_{\mathcal{T}_{k+1}}}
\\ \nonumber
&&=(f, (I-I_{k}^{^{\mathrm{CR}}})v_{k+1}^{\scriptscriptstyle{\mathrm{CR}}})
\\ \nonumber
&&\lesssim F_{R_{k}}^{1/2} \cdot \|A^{1/2}\nabla_{k+1}(I_{k+1}^{^{\mathrm{CR}}}u_0^{k+1}-J_{k+1}I_{k}^{^{\mathrm{CR}}}u_{k})\|_{\mathcal{T}_{k+1}}
\\ \nonumber
&& \leqslant F_{R_{k}}^{1/2} \|A^{1/2}(\nabla_{w, k+1}u_{k+1}- \nabla_{w, k}u_{k})\|_{\mathcal{T}_{k+1}}
\\  \label{eq:dr-2}
&&\quad + \|A^{1/2}(\nabla_{k+1} J_{k+1}I_{k}^{^{\mathrm{CR}}}u_0^{k} -\nabla_{k}I_{k}^{^{\mathrm{CR}}}u_0^{k})\|_{\mathcal{T}_{k+1}}).
\end{eqnarray}
For the second term in \eqref{eq:dr-1}, together with Lemma \ref{lem:I-projection2}, applying the Cauchy-Schwarz inequality implies
\begin{eqnarray}\nonumber
\lefteqn{(A(\nabla_{w, k+1}u_{k+1}- \nabla_{w, k}u_{k}), \nabla_{k+1}J_{k+1} I_{k}^{^{\mathrm{CR}}}u_0^{k}-\nabla_{k}I_{k}^{^{\mathrm{CR}}}u_0^{k})_{\mathcal{T}_{k+1}}}
\\ \nonumber
&&\leqslant \|A^{1/2}(\nabla_{w, k+1}u_{k+1}- \nabla_{w, k}u_{k})\|_{\mathcal{T}_{k+1}}
\\ \label{eq:dr-3}
&&\quad \cdot \|A^{1/2}(\nabla_{k+1} J_{k+1}I_{k}^{^{\mathrm{CR}}}u_0^{k} -\nabla_{k}I_{k}^{^{\mathrm{CR}}}u_0^{k})\|_{\mathcal{T}_{k+1}}.\ \hspace{2.38cm}
\end{eqnarray}

After inserting \eqref{eq:dr-2} and \eqref{eq:dr-3} into \eqref{eq:dr-1}, using the Young's inequality implies
\begin{equation}
E_k\lesssim F_{R_{k}} + \|A^{1/2}\nabla_{k+1}(J_{k+1}I_{k}^{^{\mathrm{CR}}}u_0^{k} -I_{k}^{^{\mathrm{CR}}}u_0^{k})\|_{\mathcal{T}_{k+1}}^2.
\end{equation}
At last, using Lemma \ref{lem:j-projection} leads to the desired result.
\end{proof}

\subsection{The Optimality of the AmWG}
In this section, the optimality of the AmWG Algorithm \ref{alg:amwg} will be shown. First a C\'ea-type lemma can be obtained as follows.
 
\begin{lemma}
There exists a constant $C_3$ depending only on the shape regularity of $\mathcal{T}$ such that
\begin{eqnarray}\nonumber
\lefteqn{\|A^{1/2}(\nabla u - \nabla_w u_{\mathcal{T}})\|_\mathcal{T}^2+ F(f, \mathcal{T})}\\ \label{eq:cea}
&&\leqslant
C_4\inf_{v_\mathcal{T}\in V(\mathcal{T})}\left(\|A^{1/2}(\nabla u - \nabla_w v_{\mathcal{T}})\|_{\mathcal{T}}^2 + F(f, \mathcal{T})\right).
\end{eqnarray}
\end{lemma}

\begin{proof}
The application of Strang's lemma~\cite{Ciarlet:1978Finite} yields
\begin{eqnarray}\nonumber
\lefteqn{\|A^{1/2}(\nabla u - \nabla_w u_{\mathcal{T}})\|_{\mathcal{T}}^2}
\\ \nonumber
&&=\|A^{1/2}(\nabla u - \nabla_\mathcal{T} I_{\mathcal{T}}^{^{\mathrm{CR}}} u_{0})\|_{\mathcal{T}}^2
\\  \nonumber
&&\lesssim\|A^{1/2}(\nabla u - \nabla_\mathcal{T} I_{\mathcal{T}}^{^{\mathrm{CR}}} v_{0})\|_{\mathcal{T}}^2
\\ \label{eq:cea_1}
&&\quad+\sup_{v_{\mathcal{T}}\in V(\mathcal{T})}\frac{(A\nabla u, \nabla_\mathcal{T} I_{\mathcal{T}}^{^{\mathrm{CR}}} v_{0})- (f, I_{\mathcal{T}}^{^{\mathrm{CR}}} v_{0})}{\|\nabla_\mathcal{T} I_{\mathcal{T}}^{^{\mathrm{CR}}} v_{0}\|} .
\end{eqnarray}
We need to define the following higher order conforming finite element space
\begin{equation}\label{eq:S}
V^{\mathrm{c}}(\mathcal{T}):=\{v\in H_0^1(\Omega), v|_\tau\in (P_3(K))^d, d=2, 3, \forall \tau\in \mathcal{T}\},
\end{equation}
there exists an interpolation $\Upsilon_\mathcal{T}: V^{\mathrm{nc}}(\mathcal{T})\rightarrow V^{\mathrm{c}}(\mathcal{T})$ with following properties (see~\cite[Section 6]{Hu;Xu;2013})
\begin{equation}
\begin{aligned}\label{eq:r-projection1}
&\int_e (w_\mathcal{T}-\Upsilon_\mathcal{T}w_\mathcal{T})\cdot c_e \mathrm{d}s=0, \quad \forall c_e\in P_1(e),
\\
&\int_\tau(w_\mathcal{T}-\Upsilon_\mathcal{T}w_\mathcal{T})\mathrm{d} x=0,
\end{aligned}
\end{equation}
for $w_\mathcal{T}\in V^{\mathrm{nc}}(\mathcal{T})$,  edge/face $e$ and $\tau\in \mathcal{T}$. We also have
\begin{equation}\label{eq:r-projection2}
\begin{aligned}
\|w_\mathcal{T}-\Upsilon_\mathcal{T}w_\mathcal{T}\|_{\mathcal{T}} + h_\tau \|\nabla \Upsilon_\mathcal{T}w_\mathcal{T}\|_{\mathcal{T}} 
\lesssim h_\tau \|\nabla_h w_\mathcal{T}\|_{0, \omega(\tau)}.
\end{aligned}
\end{equation}
For any $v_{\mathcal{T}}=\{v_0, v_b\}\in V^(\mathcal{T})$, the following decomposition holds:
\begin{eqnarray*}
\begin{aligned}
 \lefteqn{(A\nabla u, \nabla_\mathcal{T} I_{\mathcal{T}}^{^{\mathrm{CR}}} v_{0})- (f, I_{\mathcal{T}}^{^{\mathrm{CR}}}v_{0})}\\
 &=(A\nabla u -A\nabla_\mathcal{T} I_{\mathcal{T}}^{^{\mathrm{CR}}} u_{0}, \nabla_\mathcal{T} (I_{\mathcal{T}}^{^{\mathrm{CR}}} v_{0}-\Upsilon_\mathcal{T} I_{\mathcal{T}}^{^{\mathrm{CR}}} v_{0}))\\
 &\quad- (f, I_{\mathcal{T}}^{^{\mathrm{CR}}} v_{0}-\Upsilon_\mathcal{T} I_{\mathcal{T}}^{^{\mathrm{CR}}} v_{0})) 
+ (A\nabla_\mathcal{T} I_{\mathcal{T}}^{^{\mathrm{CR}}} u_{0}, \nabla_\mathcal{T} (I_{\mathcal{T}}^{^{\mathrm{CR}}} v_{0}-\Upsilon_\mathcal{T} I_{\mathcal{T}}^{^{\mathrm{CR}}} v_{0})), 
\end{aligned}
\end{eqnarray*}
By the properties \eqref{eq:r-projection1} and \eqref{eq:r-projection2}, we have
\begin{eqnarray}\nonumber
  \lefteqn{(A\nabla u, \nabla_\mathcal{T} I_{\mathcal{T}}^{^{\mathrm{CR}}} v_{0})- (f, I_{\mathcal{T}}^{^{\mathrm{CR}}}v_{\mathcal{T}})}
 \\ 
\label{eq:cea_2}
&&\qquad  \lesssim\|\nabla_\mathcal{T} I_{\mathcal{T}}^{^{\mathrm{CR}}} v_{0}\|  \cdot \|A^{1/2}(\nabla u - \nabla_\mathcal{T} I_{\mathcal{T}}^{^{\mathrm{CR}}} v_{0})\|_{\mathcal{T}} + F^{1/2}(f, \mathcal{T}) \cdot \|\nabla_\mathcal{T} I_{\mathcal{T}}^{^{\mathrm{CR}}} v_{0}\|.\ \ \ \ \ \ \ \ 
\end{eqnarray}
After inserting \eqref{eq:cea_2} into \eqref{eq:cea_1}, we use Young's inequality to have the desired result \eqref{eq:cea}.
\end{proof}

Let $\mathbb{T}_N$ be the set of all partitions $\mathcal{T}$ which is refined from $\mathcal{T}_0$
and $\#\mathcal{T} \leqslant N$. For a given partition $\mathcal{T}$, we introduce the following semi-norm:
\begin{equation}\label{eq:u-fnorm}
|u, f|_s^2=\sup_{N>s}N^s\inf_{\mathcal{T}\in\mathbb{T}_N}\left(\inf_{v_\mathcal{T}\in V(\mathcal{T})}\|A^{1/2}(\nabla u -\nabla_w v_{\mathcal{T}})\|_\mathcal{T}^2 + F(f, \mathcal{T}) \right),
\end{equation}
and the approximation class is then defined as follows, for $s > 0$:
\begin{equation}\label{eq:As}
\mathbb{A}_s:=\{(u, f): |u, f|_s< +\infty\}.
\end{equation}

In this case, we recall all ingredients needed for the optimality of the adaptive procedure:
\begin{description}
  \item[(1)] Quasi-orthogonality in Lemma \ref{lem:quasi-orthogonality}:
  \begin{eqnarray*}
 (A(\nabla u-\nabla_{w,k+1}u_{k+1}), \nabla_{w,k+1}u_{k+1}-\nabla_{w,k}u_k)_{\mathcal{T}_{k+1}}
 \leqslant \sqrt{C_1} F_{R_k}^{1/2}\cdot \varepsilon_{k+1};
  \end{eqnarray*}
  \item[(2)] Discrete Reliability in Lemma \ref{lem:discrete-reliability}:
\begin{eqnarray*}
E_k\leqslant C_{dr} \eta_{R_k}^2;
  \end{eqnarray*}
  \item[(3)] The lower bound when the polynomial degree $\ell=1$ : 
\begin{eqnarray*}
C_{L}\eta(\nabla_w \uH, \mathcal{T})^2 \leqslant 
\|A^{1/2}(\nabla u-\nabla_w \uH)\|_{\mathcal{T}}^2 + 
F(f, \mathcal{T})
  \end{eqnarray*}
\end{description}

Thanks to these preparations, following~\cite{Huang;Xu;2013}, the optimality result is as follows:
\begin{theorem}
Let $u$ be the solution of \eqref{eq:pb-model}, $\{\mathcal{T}_N, u_N, \eta(\nabla_w u_k,\mathcal{T}_N)\}_{N\geq 0}$ be a sequence of meshes, finite element approximations and error estimators produced by Algorithm \ref{alg:amwg}. For $(u, f)\in \mathbb{A}_s$ with
\begin{equation}
\theta\in\left(0, \frac{\min(1, C_L)}{\min(1, C_L)+C_1+1}\right),
\end{equation}
then it holds that
\begin{equation} 
\|A^{1/2}(\nabla u-\nabla_w u_{\mathcal{T}_N})\|_{\mathcal{T}_N}^2 + F(f, \mathcal{T}_N) \lesssim |u, f|_s^2(\#\mathcal{T}_N-\#\mathcal{T}_0)^{-2s}.
\end{equation}
\end{theorem}

\section{Numerical Experiments}
\label{sec:numerics}
In this section, with the aid of the MATLAB software package iFEM~\cite{Chen2009ifem}, we implement the following numerical experiments to verify the convergence and quasi-optimality of the Algorithm \ref{alg:AmWG}. 

\begin{example}\label{exa:2}
In this example, we choose a square  domain $\Omega = (-1, 1)^2$ and coefficient $A = \mathbf{I}$, the exact solution of \eqref{eq:pb-model} is $u(x, y)=\frac{y(x^2-1)(y^2-1)}{x^2+y^2+0.01}$ . 
\end{example} 

\RV{On the left of Figure \ref{fig:AmWG.2} shows the initial mesh $\mathcal{T}_0$ for Example \ref{exa:2}; on the right of Figure \ref{fig:AmWG.2} shows the refined mesh after $k=18$ iterations for the Example \ref{exa:2} with $\theta=0.5$. }

\begin{figure}[H]
	\centering
	\subfigure{
		\begin{minipage}[t]{0.4\linewidth}
			\centering
			\includegraphics[width=2.3in]{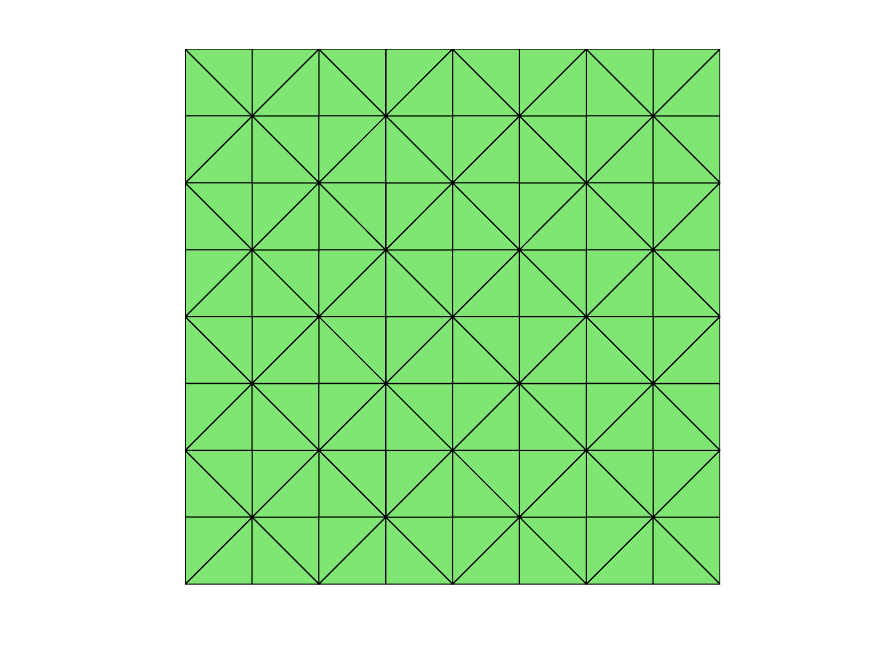}
	\end{minipage}}
	\subfigure{
		\begin{minipage}[t]{0.4\linewidth}
			\centering
			\includegraphics[width=2.3in]{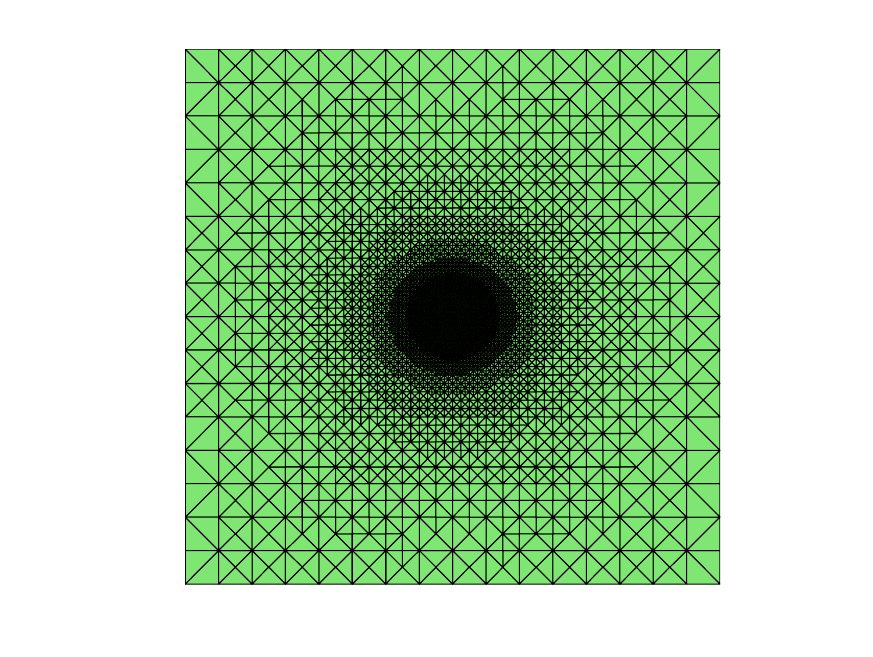}
		\end{minipage}%
	}%
	\centering
	\caption{\small{The initial mesh (left); Adaptively refined mesh after $k=18$ iterations(right)} for Example \ref{exa:2}.}\label{fig:AmWG.2}
\end{figure}

Figure \ref{fig:error2} shows the rate of $\ln\# \mathcal{T}_k$ v.s. $\ln \|A^{1/2}(\nabla u-\nabla_w u_{\mathcal{T}_k})\|_{\mathcal{T}_{k}}$ with different marking parameters $\theta =0.3, 0.5$ and $0.7$, where $\#\mathcal{T}_k$  and $u_{{k}}$ represent the number of elements and the corresponding solution, respectively, gotten from the Algorithm \ref{alg:AmWG}.

\begin{figure}[H]
	\centering
	\includegraphics[width=4in]{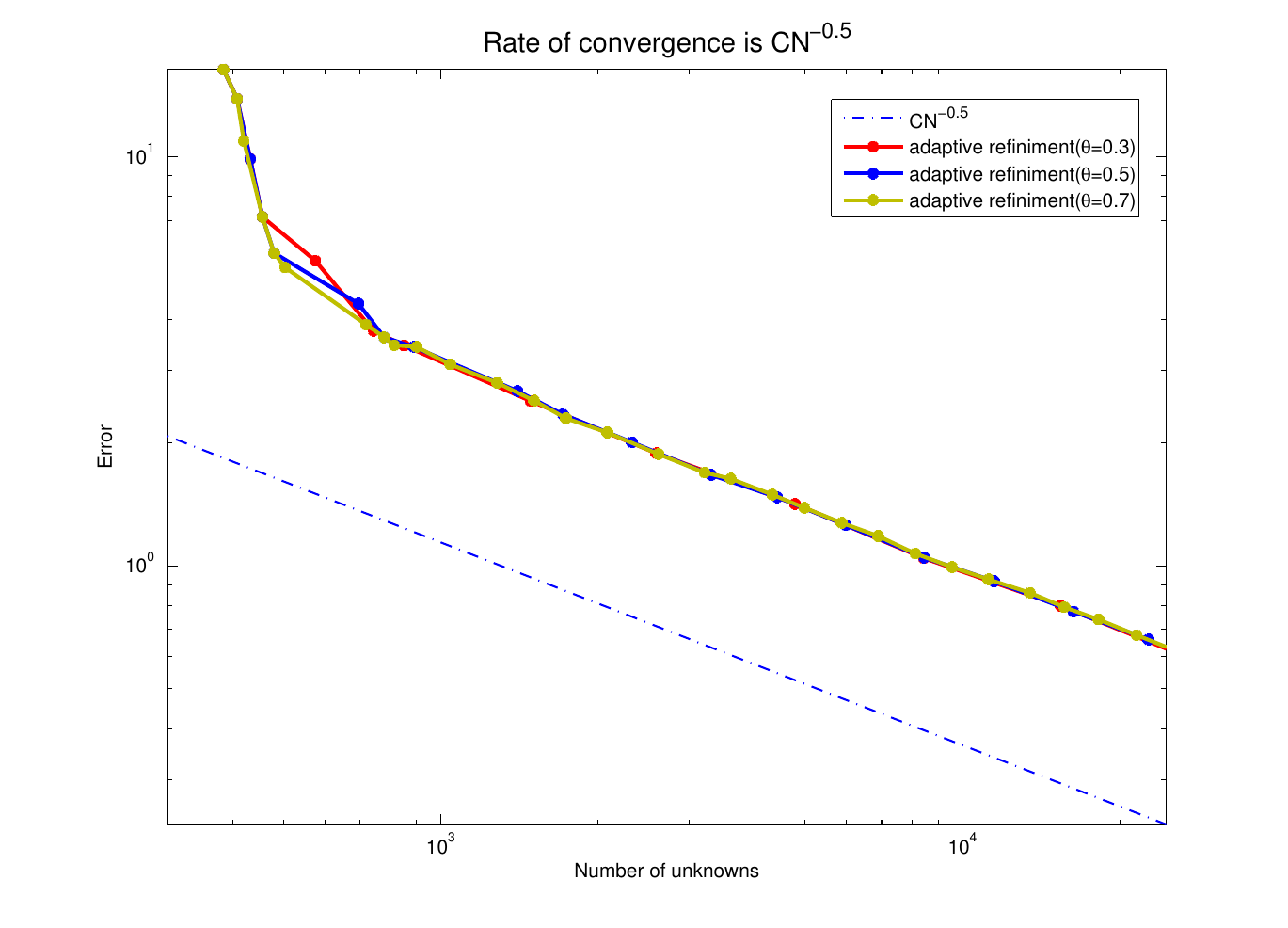}
	\caption{\small{Quasi-optimality of the adaptive mesh refinements with marking parameters $\theta=0.3, 0.5, 0.7$}.}\label{fig:error2}
\end{figure}

\begin{example}\label{exa:1}
In this example, we choose the L-shape domain $\Omega = (-1, 1)^2/([0, 1)\times(-1,0])$ and coefficient $A = \mathbf{I}$, the exact solution of \eqref{eq:pb-model} is $u(x, y)=r^{2/3}\sin(\frac{2\theta}{3})$. 
\end{example}

\RV{On the left of Figure \ref{fig:AmWG.1} shows the initial mesh $\mathcal{T}_0$ for Example \ref{exa:1};  on the right of Figure \ref{fig:AmWG.1} shows the refined mesh after $k=20$ iterations for the Example \ref{exa:1} with $\theta=0.7$.}

\begin{figure}[htbp]
	\centering
	\subfigure{
		\begin{minipage}[t]{0.4\linewidth}
			\centering
			\includegraphics[width=2.3in]{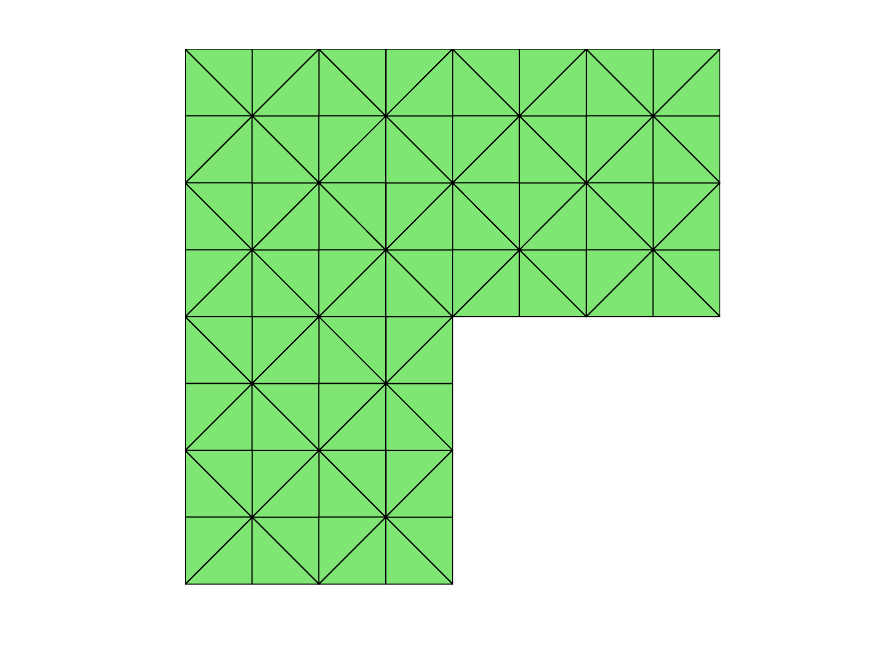}
	\end{minipage}}
	\subfigure{
		\begin{minipage}[t]{0.4\linewidth}
			\centering
			\includegraphics[width=2.3in]{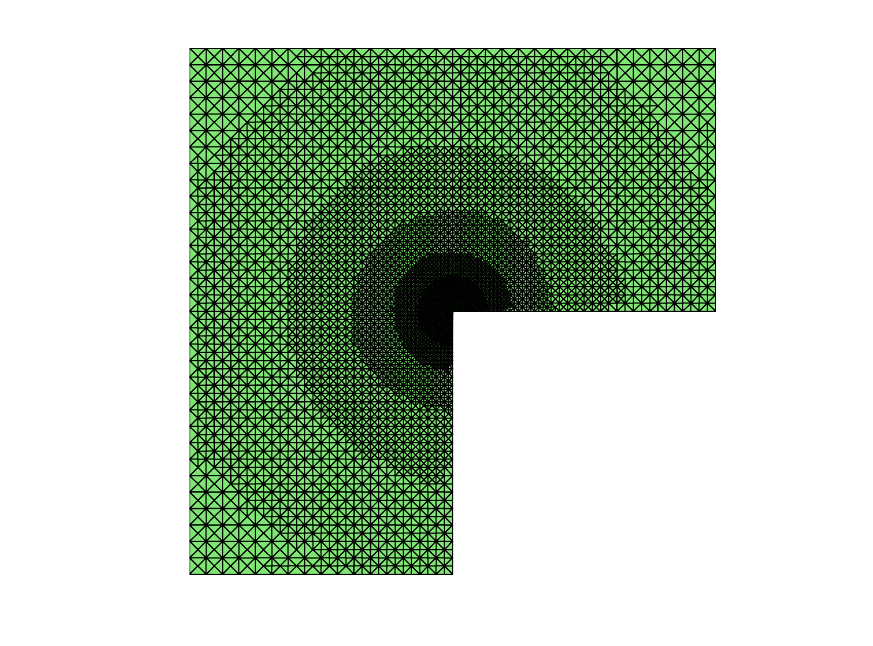}
		\end{minipage}%
	}%
	\centering
	\caption{\small{The initial mesh (left); Adaptively refined mesh after $k=20$ iterations(right)} for Example \ref{exa:1}.}\label{fig:AmWG.1}
\end{figure}

Figure \ref{fig:error} shows the rate of $\ln\# \mathcal{T}_k$ v.s. $\ln \|A^{1/2}(\nabla u-\nabla_w u_{\mathcal{T}_k})\|_{\mathcal{T}_{k}}$ with different marking parameters $\theta =0.3, 0.5$ and $0.7$, where $\#\mathcal{T}_k$  and $u_{{k}}$ represent the number of elements and the corresponding solution, respectively, gotten from the Algorithm \ref{alg:AmWG}.

\begin{figure}[htbp]
	\centering
	\includegraphics[width=4in]{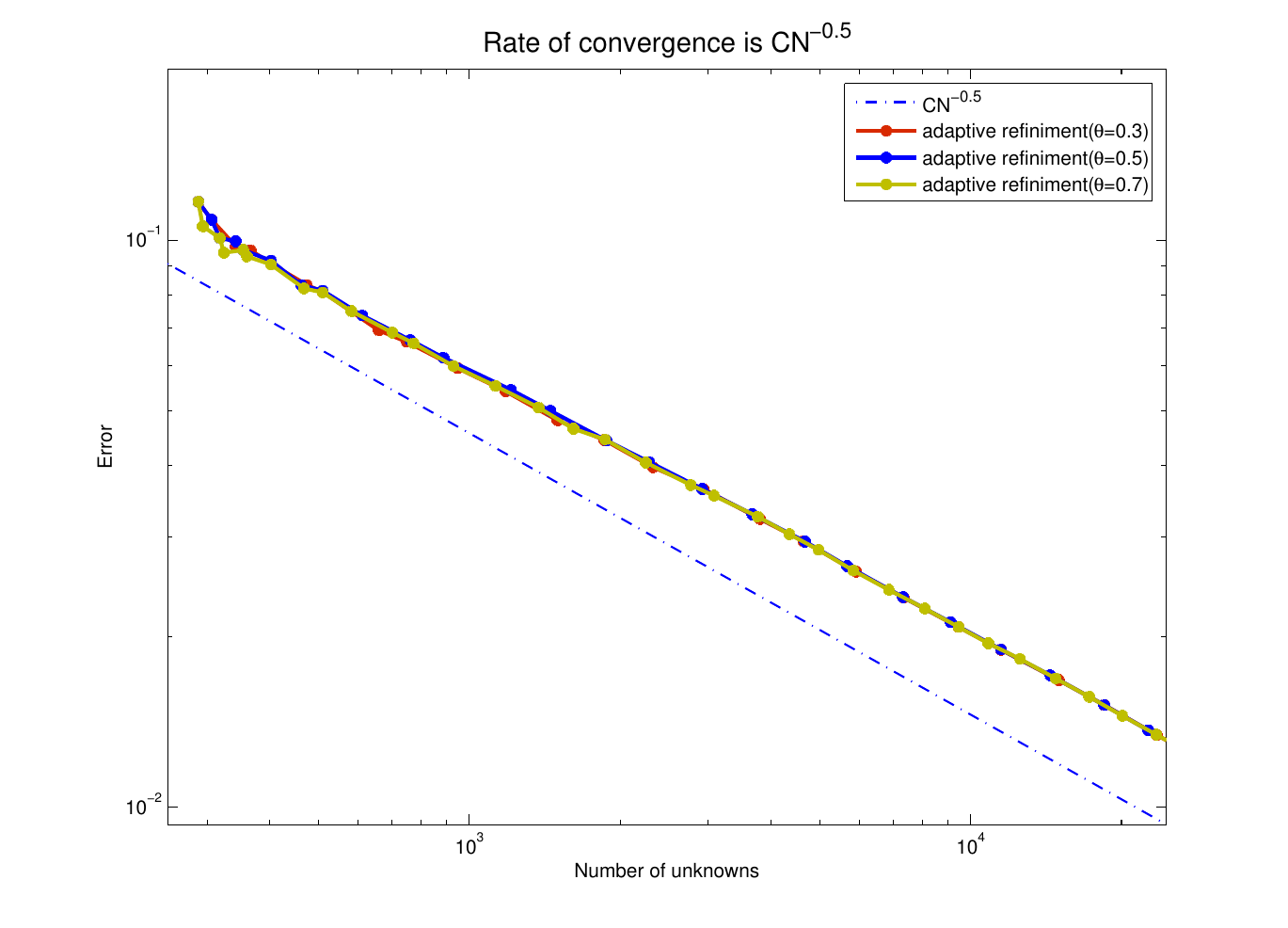}
	\caption{\small{Quasi-optimality of the adaptive mesh refinements with marking parameters $\theta=0.3, 0.5, 0.7$}.}\label{fig:error}
\end{figure}

The right one of Figures  \ref{fig:AmWG.2} and \ref{fig:AmWG.1} show that the mesh is locally refined. And the curves in Figures  \ref{fig:error2} and \ref{fig:error} indicate that the convergence and the quasi-optimality  of the 
Algorithm \ref{alg:AmWG}, namely
$$
\|A^{1/2}(\nabla u-\nabla_w u_{\mathcal{T}_k})\|_{\mathcal{T}_{k}} \leq C(\#\mathcal{T}_k)^{-1/2}.
$$

\end{document}